\documentclass{amsart}

\usepackage{xcolor}
\usepackage{amsmath, amsthm, amssymb, amsfonts, mathtools, enumerate}
\usepackage{hyperref}
\usepackage{verbatim} 
\usepackage{graphicx,epsfig,color}
\usepackage{scrpage2}
\usepackage{exscale}

\theoremstyle{plain}

\newtheorem{theorem}{Theorem}[section]

\newtheorem{lemma}[theorem]{Lemma}

\newtheorem{corollary}[theorem]{Corollary}
\newtheorem{conjecture}[theorem]{Conjecture}
\newtheorem{proposition}[theorem]{Proposition}

\theoremstyle{definition}

\newtheorem{definition}[theorem]{Definition}

\theoremstyle{remark}
\newtheorem{remark}[theorem]{Remark}

\newtheorem{example}[theorem]{Example}

\DeclareSymbolFont{AMSb}{U}{msb}{m}{n}
\DeclareMathSymbol{\N}{\mathalpha}{AMSb}{"4E}
\DeclareMathSymbol{\R}{\mathalpha}{AMSb}{"52}
\DeclareMathSymbol{\Z}{\mathalpha}{AMSb}{"5A}
\DeclareMathSymbol{\D}{\mathalpha}{AMSb}{"44}
\DeclareMathSymbol{\s}{\mathalpha}{AMSb}{"53}

\DeclareMathOperator{\md}{md}

\renewcommand{\r}{{\footnotesize \ensuremath{\mathbb{R}}}}

\DeclareMathOperator{\supp}{supp}

\DeclareMathOperator{\m}{m}

\DeclareMathOperator{\diam}{diam}

\newcommand{\T}{\mathcal{T}}

\DeclareMathOperator{\OptGeo}{OptGeo}

\def\co{\colon\thinspace}
\def\eps{\epsilon}

\title{On gluing Alexandrov spaces with lower Ricci curvature bounds}
\author{Vitali Kapovitch}
\address{University of Toronto}
\email{vkt@math.toronto.edu}
\author{Christian Ketterer}

\address{University of Toronto}
\thanks{
CK is funded by the Deutsche Forschungsgemeinschaft (DFG, German Research Foundation) -- Projektnummer 396662902. KTS gratefully acknowledges financial support by the European Union through the ERC-AdG ``RicciBounds''
and by the DFG through the Excellence Cluster ``Hausdorff Center for Mathematics'' and through the Collaborative Research Center 1060.}
\thanks{{\it 2010 Mathmatics Subject Classification.} Primary 53C21, 54E35. Keywords: metric measure space, curvature-dimension condition, gluing construction.}
\email{ckettere@math.toronto.edu.}
\author{Karl-Theodor Sturm}
\address{University of Bonn}
\email{sturm@iam.uni-bonn.de}
\begin{document}
\begin{abstract}
In this paper we prove that in the class of metric measure space with Alexandrov curvature bounded from below the Riemannian curvature-dimension condition $RCD(K,N)$ with $K\in \R$ \& $N\in [1,\infty)$ is preserved under doubling and gluing constructions.
\end{abstract}
\maketitle
\tableofcontents
\section{Introduction and Statement of Main Results}
A way to construct Alexandrov spaces is by gluing together two or more given Alexandrov spaces along isometric connected components of their
intrinsic boundaries. The isometry between the boundaries is understood w.r.t. induced length metric. A special case of this construction is the double space where one glues together two copies of the same Alexandrov space with nonempty boundary. 
It was shown by Perelman that the double of an  Alexandrov space of curvature  $\ge k$ is again Alexandrov of curvature  $\ge k$. Petrunin later showed  \cite{pg}  that the lower curvature bound is preserved in general for any gluing of two possibly different Alexandrov spaces.

In this article we study Ricci curvature bounds in the sense of Lott, Sturm and Villani for this setup. More precisely, we consider the class of $n$-dimensional Alexandrov spaces with some lower curvature bound equipped with a Borel measure of the form $\Phi \mathcal H^n=\m$ for a semi-concave function $\Phi:X\rightarrow [0,\infty)$ such that the corresponding metric measure space $(X,d,\m)$ satisfies a curvature-dimension condition $CD(K,N)$ for $K\in \R$ and $N\in [n,\infty)$. Here $K$ does not necessarily coincide with $k(n-1)$.  In particular, it's possible that $k<0$ but $K\ge 0$.

To state our main theorem we recall the following. The Alexandrov boundary of $(X,d)$ is denoted as $\partial X$ equipped with the induced length metric $d_{\partial X}$. We write $\Sigma_p$ for the space of direction at $p\in X$ that is an Alexandrov space with curvature bounded below by $1$. 
We say
$v\in \Sigma_p$ for $p\in \partial X$ is a normal vector at $p$ if $\angle(v,w)=\frac{\pi}{2}$ for any $w\in \partial \Sigma_p$.
Here $d_x\Phi_i$ denotes the differential of the semi-concave function $\Phi$ at some point $x\in X_i$. We also refer to the remarks after Definition \ref{def:semiconcavefct}.

Our main theorem is 
\begin{theorem}[Glued spaces]\label{main} 
For $i=0,1$ let $X_i$ be $n$-dimensional Alexandrov spaces with curvature bounded below and let $\m_{X_i}=\Phi_i \mathcal H^n_{X_i}$ be measures where $\Phi_i:X_i\rightarrow [0,\infty)$ are  semi-concave functions. Suppose there exists an isometry $\mathcal I:\partial X_0\rightarrow \partial X_1$ such that $\Phi_0=\Phi_1\circ \mathcal I$. 

If the metric measure spaces $(X_i, d_{X_i}, \m_i)$ satisfy the curvature-dimension condition $CD^*(K,N)$ for $K\in \R$, $N\in [1,\infty)$ and if $$d_p\Phi_0(v_0)+d_p\Phi_1(v_1)\leq 0\mbox{ $\forall p\in \partial X_i$ and any normal vectors $v_i\in \Sigma_pX_i$, $i=0,1$},$$ then the glued metric measure space $(X_0\cup_{\mathcal I} X_1, (\iota_0)_{\#}\m_{X_0}+(\iota_1)_{\#}\m_{X_1}))$ satisfies the reduced curvature-dimension condition $CD^*(K,N)$.
\end{theorem}
\begin{remark}
If the measures are finite, one can replace in the conclusion of Theorem \ref{main} the condition $CD^*(K,N)$ with the full curvature-dimension condition $CD(K,N)$. In this case the two conditions are equivalent \cite{cavmil}.
\end{remark}
\begin{corollary}\label{maincor}
For $i=0,1$ let $X_i$ be Alexandrov spaces with curvature bounded below, and let $\mathcal I:\partial X_0\rightarrow \partial X_1$ be an isometry. Assume the metric measure spaces $(X_i,d_{X_i},\mathcal H^n_{X_i})$ satisfy the condition  $CD^*(K,N)$
for $K\in \R, N\in [1,\infty)$. 

Then the metric measure space $(X_0\cup_{\mathcal I} X_1, \mathcal H^n_{X_0\cup_{\mathcal I} X_1})$ satisfies the condition $CD^*(K,N)$.
\end{corollary}
\begin{remark}
An Alexandrov space with curvature bounded from below is infinitesimally Hilbertian. Therefore it satisfies the condition $CD(K,N)$ (or $CD^*(K,N)$) if only if it satisfies the Riemannian curvature-dimension condition $RCD(K,N)$ (or $RCD^*(K,N)$) (Corollary \ref{cor:rcd}).
\end{remark}
\begin{remark}{If $X_i$ are convex domains in smooth Riemannian manifolds with lower Ricci curvature bounds, then the  statement of Corollary \ref{maincor} is regarded as folklore. A complete proof has been given in \cite{PS}, based on a detailed approximation property derived in \cite{sch}.}
\end{remark}

For general  noncollapsed $RCD$ spaces there  are two natural notions of boundary: one that was introduced by DePhillippis and Gigli in \cite{GP-noncol} and another by Mondino and the first named author in ~\cite{Kap-Mon19}. 
Conjecturally both notions coincide and the boundary (defined either way) is a closed subset in the ambient space.
We make the following conjecture
\begin{conjecture}
For $i=0,1$ let $X_i$ be noncollapsed $RCD(K,n)$ spaces with nonempty boundary $\partial X_i$.  Suppose there exists an isometry $\mathcal I:\partial X_0\rightarrow \partial X_1$.  Then the glued  metric measure space $(X_0\cup_{\mathcal I} X_1, \mathcal H^n_{X_0\cup_{\mathcal I} X_1})$ satisfies the condition $RCD(K,n)$. 
\end{conjecture}
As a biproduct of our proof Theorem \ref{main} we also obtain the following result that also seems to be new.

\begin{theorem}\label{th:concavity}
For $i=0,1$ let $X_i$ be $n$-dimensional Alexandrov spaces with curvature bounded below as in the previous theorem, let $X_0\cup_{\mathcal I} X_1$ be the glued Alexandrov spaces
and let $\Phi_i:X_i\rightarrow \R$, $i=0,1$, be semi-concave with $\Phi_0|_\partial X_0=\Phi_1|_\partial X_1$ such that for any $p\in \partial X_i$ it holds that
\begin{align*}
d\Phi_0|_p(v_0)+ d\Phi_1|_p(v_1)\leq 0\ \ \forall\mbox{ normal vectors }v_i\in \Sigma_pX_i,\  i=0,1.
\end{align*}
Then $\Phi_0+\Phi_1: X_0\cup_{\mathcal I} X_1 \rightarrow \R$ is semiconcave. 
\end{theorem}
We say that
a function $\Phi:X\rightarrow \R$ on an Alexandrov space $X$ is double semi-concave if $\Phi\circ P:\hat X\rightarrow \R$ is semi-concave in the usual sense where $\hat X$ denotes the Alexandrov double space of $X$ and $P:\hat X\rightarrow X$ is the canoncial map. We give an alternative characterisation of this condition in Lemma \ref{gluingcondition} and Corollary \ref{cor:gluingfct}.

As another consequence of our main theorem we also obtain the following.
\begin{corollary}[Doubled spaces]\label{main2}
Let $X$ be an $n$-dimensional Alexandrov space with curvature bounded below, and let $\m_X=\Phi \mathcal H^n_X$ be a measure for a double semi-concave function $\Phi: X\rightarrow [0,\infty)$. Assume the metric measure space $(X,d_X,\m_X)$ satisfies the condition $CD^*(K,N)$ for $K\in \R$ and $N\geq 1$. Then, the double space
 $(\hat X, d_{\hat X}, \m_{\hat X})$ satisfies the condition $CD^*(K,N)$.
\end{corollary}
{
Let us  briefly comment on the statement and  the proof of Theorem \ref{main} and Corollary \ref{maincor}. By Petrunin's glued space theorem one knows that the glued space of two Alexandrov spaces with curvature bounded from below is again an Alexandrov space with the same lower  curvature bound. However the intrinsic {\it best lower Ricci bound} might be different from the Alexandrov curvature bound. So Petrunin's theorem does not imply any of the statements above. 

But we can use the improved regularity of the glued space for our purposes. It implies {\it some} lower Ricci bound that yields a priori information for transport densities and densities along needles in the Cavalletti-Mondino $1D$ localisation procedure. At this point a crucial difficulty appears. It is not known whether geodesics cross the boundary set where the spaces are glued together, only finitely many times.  This difficulty does not occur for the double space construction. By symmetry in this case it is known that geodesics in the double space only cross once. 

For general glued spaces we overcome this problem by the following strategy.  First, given a $1D$ localisation we show that the collection of  geodesic  that cross the boundary infinitely many times has measure $0$ w.r.t. to the corresponding quotient measure. Then, we apply a theorem of Cavalletti and Milman on characterization of synthetic Ricci curvature bounds via $1D$ localisation.
}

\begin{remark} 
As pointed out by Rizzi \cite{Ri17},
 in the previous theorem one cannot replace the curvature-dimension conditon $CD(K,N)$ for any $K\in \R$ and $N\in (1,\infty)$ with the measure contraction property $MCP(K,N)$ \cite{stugeo2, ohtamcp}. The $MCP$ is a weaker condition that still characterizes lower Ricci curvature bounds for $N$-dimensional smooth manifolds and is also consistent with lower Alexandrov curvature bounds. For the precise definition we refer to  \cite{stugeo2, ohtamcp}. 
 The counterexample in \cite{Ri17} is given by the Grushin half-plane which satisfies $ MCP(0,N)$ if and only if $N\ge4$ while its double satisfies $ MCP(0,N)$ if and only if $N\ge 5$.

 Another example that is even  Alexandrov is provided in the last section of this article (Example \ref{example}).
\end{remark}
\begin{remark}
Let us mention that one can show that Petrunin's gluing theorem holds for gluing $n$-dimensional Alexandrov spaces along isometric extremal subsets of codimension 1 which do not need to be equal to the whole components of their boundaries (see for instance \cite{mitsuishi}). For example gluing two triangles having a side of equal length with all adjacent angles to it $\le \pi/2$ is again an Alexandrov space (a convex quadrilateral). Our results then generalize to this situation as well.  
\end{remark}

\subsection{Application to heat flow with Dirichlet boundary condition}

The concept of doubling has recently found significant application in the study of the heat flow with Dirichlet boundary conditions. In particular, it allows the use of optimal transportation techniques.
As widely known, these techniques are not directly applicable since the Dirichlet heat flow will not preserve masses.

As observed in \cite{PS}, this obstacle can be overcome by looking at the heat flow in the doubled space instead. The latter is accessible to optimal transport techniques and to the powerful theory of
metric measure spaces with synthetic Ricci bounds.
Moreover, it can always be expressed as a linear combination of the Dirichlet heat flow and the Neumann heat flow on the original space -- and vice versa, both the Dirichlet and the Neumann heat flow on the original space can be expressed in terms of the heat flow on the doubled space.

More precisely now, let $X$ be an $n$-dimensional Alexandrov space with curvature bounded below, and let $\m_X=\Phi \mathcal H^n_X$ be a measure for a double semi-concave function $\Phi: X\rightarrow [0,\infty)$. 
Let $(P_t)_{t\ge0}$ denote the heat semigroup with Neumann boundary conditions on $X$ and let $(P^0_t)_{t\ge0}$ denote the heat semigroup on $X^0:=X\setminus \partial X$ with Dirichlet boundary conditions with respective generators $\Delta$ and $\Delta^0$.
\begin{theorem}
Assume the metric measure space $(X,d_X,\m_X)$ satisfies the condition $CD(K,\infty)$ for $K\in \R$. Then the following gradient estimate of Bakry-Emery type
\begin{equation}\label{grad}\big|\nabla P_t^0f\big|\le e^{-Kt}\, P_t\big|\nabla f\big|\quad\text{a.e.~on }X^0\end{equation}
and the following Bochner inequality hold true 
\begin{equation}\label{boch}\frac12\Delta\, \big|\nabla f\big|^2-\big\langle\nabla f, \nabla\Delta^0 f\big\rangle \ge K\, \big|\nabla f\big|^2\end{equation}
weakly on $X^0$ for all sufficiently smooth $f$ on $X$. (Note that in the latter estimate, two different Laplacians appear and in the former, two different heat semigroups.)
{More precisely, \eqref{grad} holds for all $f\in W^{1,2}_0(X^0)$, the form domain for the Dirichlet Laplacian. And  \eqref{boch}  is rigorously formulated as 
\begin{equation*}\frac12\int_{X^0} \Delta\varphi\, \big|\nabla f\big|^2\,d\m-\iint_{X^0}\varphi\,\big\langle\nabla f, \nabla\Delta^0 f\big\rangle\,d\m \ge K\, \int_{X^0}\varphi\, \big|\nabla f\big|^2\,d\m\end{equation*}
for all $f\in D(\Delta^0)$ with $\Delta^0f\in W^{1,2}_0(X^0)$ and all nonnegative $\varphi\in D(\Delta^0)$ with $\varphi, \Delta^0\varphi\in L^\infty$.}
\end{theorem}

\begin{proof} Both estimates follow from Corollary \ref{main2} and \cite{PS}, Thm. 1.26. {
For the readers' convenience, let us briefly recall the main argument. The estimates for the Dirichlet heat semigroup and Dirichlet Laplacian are direct consequences of analogous estimates for the heat semigroup $(\hat P_t)_{t\ge0}$ and Laplacian $\hat\Delta$ on the doubled space $$\hat X:=X\cup X'\Big/_{\partial X=\partial X'}$$ obtained by gluing $X$ and a copy of it, say $X'$, along their common boundary $\partial X\sim\partial X'$. Then Dirichlet and Neumann heat semigroups on $X$ can be expressed in terms of the heat semigroup on $\hat X$ as 
$$P_t^0f=\hat P_t(f-f'), \qquad P_tf=\hat P_t (f+f')$$ for any given bounded, measurable $f:X\to\R$
where $f$ is extended to $\hat X$ by putting $f:=0$ on $\hat X\setminus X$ and where $f':\hat X\to\r$ is defined as  $f'(x'):=f(x)$ if 
$x'\in X'$ denotes the mirror point of $x\in X$. 
Then the gradient estimate for $\hat P_t$ on $\hat X$ obviously implies that
$$|\nabla P^0_tf|=|\nabla \hat P_t(f-f')|\le e^{-Kt}\, \hat P_t|\nabla (f-f')|=e^{-Kt}\, P_t|\nabla f|$$
for every $f\in W^{1,2}_0(X^0)$.

Actually, \eqref{grad} is stated in  \cite{PS} only for functions $f\in W^{1,2}_0(X^0)$ which in addition satisfy
$f, |\nabla f|\in L^1$.
 But any $f \in W^{1,2}_0(X^0)$ can be approximated in $W^{1,2}$-norm by compactly supported Lipschitz functions $f_n$ (which in particular satisfy $f_n, |\nabla f_n|\in L^1$). Hence, $P_t  |\nabla f|$ is the $L^2$-limit of $P_t  |\nabla f_n|$ and the claim follows by passing to a suitable subsequence which leads to a.e.-convergence.
}\end{proof}

\subsubsection*{We outline the remaining content of the article.}

In section 2 we recall preliminaries and basics on optimal transport, Ricci curvature for metric measure spaces, Alexandrov spaces, gluing of Alexandrov spaces and  1D localisation technique. We also state a new result by Cavalletti and Milman on characterizing the Ricci curvature bounds via $1D$ localisation.

In section 3 we will give two application of the $1D$ localisation technique. The first application shows that almost all geodesics avoid set of $\mathcal H^n$-measure $0$ in the boundary in the glued space. The second application shows that given a $1D$ localisation w.r.t. an arbritrary $1$-Lipschitz function, geodesics that are tangential to the boundary have measure $0$ w.r.t. the corresponding quotient measure.

In section 4 we use the results of the previous section to prove Theorem \ref{th:concavity}.

In section 5 we prove the glued space theorem applying the results we obtained in section 3 and section 4. 
\subsubsection*{Acknowledgments}
The authors want to thank Anton Petrunin for helpful conversations on  gluing spaces and other topics.
\color{black}
\section{Preliminaries}
\subsection{Curvature-dimension condition}
Let $(X,d)$ be a complete and separable metric space equipped with a locally finite Borel measure $\m$. We call a triple $(X,d,\m)$ a metric measure space. 

A geodesic is a length minimizing curve $\gamma:[a,b]\rightarrow X$.
We denote the set of constant speed geodesics $\gamma:[a,b]\rightarrow X$ with $\mathcal G^{[a,b]}(X)$ equipped with the topology of uniform convergence and set $\mathcal G^{[0,1]}(X)=:\mathcal G(X)$. For $t\in [a,b]$ the evaluation map $e_t:\mathcal G^{[a,b]}(X)\rightarrow X$ is defined as $\gamma\mapsto \gamma(t)$ and $e_t$ is continuous.

A set of geodesics $F\subset \mathcal{G}(X)$ is said to be {\it non-branching} if  $\forall\epsilon\in (0,1)$ the map $e_{[0,\epsilon]}|_{F}$ is one to one.

The set of (Borel) probability measure is denoted with $\mathcal P(X)$, the subset of probability measures with finite second moment is $\mathcal P^2(X)$,  the set of probability measures in $\mathcal P^2(X)$ that are $\m$-absolutely continuous is denoted with $\mathcal P^2(X,\m)$ and the subset of measures in $\mathcal P^2(X,\m)$ with bounded support is denoted with $\mathcal{P}_b^2(X,\m)$.

The space $\mathcal P^2(X)$ is equipped with the $L^2$-Wasserstein distance $W_2$. 
A dynamical optimal coupling is a probability measure $\Pi\in \mathcal P(\mathcal G(X))$  such that $t\in [0,1]\mapsto (e_t)_{\#}\Pi$ is a $W_2$-geodesic in $\mathcal P^2(X)$. 
The set of  dynamical optimal couplings $\Pi\in \mathcal P(\mathcal G^{}(X))$ between $\mu_0,\mu_1\in \mathcal P^2(X)$ is denoted with $\OptGeo(\mu_0,\mu_1)$. 

A metric measure space $(X,d,\m)$ is called \textit{essentially nonbranching} if for any pair $\mu_0,\mu_1\in \mathcal P^2(X,\m)$ any $\Pi\in \OptGeo(\mu_0,\mu_1)$ is concentrated on a set of nonbranching geodesics.
\begin{definition}
For $\kappa\in \mathbb{R}$ we define $\cos_{\kappa}:[0,\infty)\rightarrow \mathbb{R}$ as the solution of 
\begin{align*}
v''+\kappa v=0 \ \ \ v(0)=1 \ \ \& \ \ v'(0)=0.
\end{align*}
$\sin_{\kappa}$ is defined as solution of the same ODE with initial value $v(0)=0 \ \&\ v'(0)=1$. That is 
\begin{align*}
\cos_{\kappa}(x)=\begin{cases}
 \cosh (\sqrt{|\kappa|}x) & \mbox{if } \kappa<0\\
1& \mbox{if } \kappa=0\\
\cos (\sqrt{\kappa}x) & \mbox{if } \kappa>0
                \end{cases}
                \quad
   \sin_{\kappa}(x)=\begin{cases}
\frac{ \sinh (\sqrt{|\kappa|}x)}{\sqrt{|\kappa|}} & \mbox{if } \kappa<0\\
x& \mbox{if } \kappa=0\\
\frac{\sin (\sqrt{\kappa}x)}{\sqrt \kappa} & \mbox{if } \kappa>0
                \end{cases}                 
                \end{align*}
Let $\pi_\kappa$ be the diameter of a simply connected space form $\mathbb S^2_k$ of constant curvature $\kappa$, i.e.
\[
\pi_\kappa= \begin{cases}
 \infty \ &\textrm{ if } \kappa\le 0\\
\frac{\pi}{\sqrt \kappa}\ &  \textrm{ if } \kappa> 0

\end{cases}
\]
For $K\in \mathbb{R}$, $N\in (0,\infty)$ and $\theta\geq 0$ we define the \textit{distortion coefficient} as
\begin{align*}
t\in [0,1]\mapsto \sigma_{K,N}^{(t)}(\theta)=\begin{cases}
                                             \frac{\sin_{K/N}(t\theta)}{\sin_{K/N}(\theta)}\ &\mbox{ if } \theta\in [0,\pi_{K/N}),\\
                                             \infty\ & \ \mbox{otherwise}.
                                             \end{cases}
\end{align*}
Note that $\sigma_{K,N}^{(t)}(0)=t$.
Moreover, for $K\in \mathbb{R}$, $N\in [1,\infty)$ and $\theta\geq 0$ the \textit{modified distortion coefficient} is defined as
\begin{align*}
t\in [0,1]\mapsto \tau_{K,N}^{(t)}(\theta)=\begin{cases}
                                            \theta\cdot\infty \ & \mbox{ if }K>0\mbox{ and }N=1,\\
                                            t^{\frac{1}{N}}\left[\sigma_{K,N-1}^{(t)}(\theta)\right]^{1-\frac{1}{N}}\ & \mbox{ otherwise}
                                           \end{cases}\end{align*}
where our convention is $0\cdot \infty =0$. 
{It holds that 
\begin{align}\label{ineq:dist}
\tau_{K,N}^{(t)}(\theta)\geq \sigma_{K,N}^{(t)}(\theta).
\end{align}}
\end{definition}
%
%\begin{lemma}[\cite{stugeo2}]\label{lem:coef}
%For all $K, K'\in \R$, all $N, N'\in (0,\infty)$, all $t\in [0,1]$ and all $\theta\in (0,\infty)$, it holds
%\begin{align*}
%\sigma_{K,N}^{(t)}(\theta)^N \sigma_{K',N'}^{(t)}(\theta)^{N'}\geq \sigma_{K+K',N+N'}^{(t)}(\theta)^{N+N'},
%\end{align*}
%and for $N\geq 1$
%\begin{align*}
%\tau_{K,N}^{(t)}(\theta)^N \sigma_{K',N'}^{(t)}(\theta)^{N'}\geq \tau_{K+K',N+N'}^{(t)}(\theta)^{N+N'}.
%\end{align*}
%\end{lemma}
\begin{definition}[\cite{stugeo2,lottvillani, bast}]
A metric measure space $(X,d,\m)$ satisfies the \textit{curvature-dimension condition} $CD(K,N)$ for $K\in \mathbb{R}$, $N\in [1,\infty)$ if for every pair $\mu_0,\mu_1\in \mathcal{P}_b^2(X,\m)$ 
there exists an $L^2$-Wasserstein geodesic $(\mu_t)_{t\in [0,1]}$ and an optimal coupling $\pi$ between $\mu_0$ and $\mu_1$ such that 
\begin{align}\label{ineq:cd}
S_N(\mu_t|\m)\leq -\int \left[\tau_{K,N}^{(1-t)}(\theta)\rho_0(x)^{-\frac{1}{N}}+\tau_{K,N}^{(t)}(\theta)\rho_1(y)^{-\frac{1}{N}}\right]d\pi(x,y)
\end{align}
where $\mu_i=\rho_id\m$, $i=0,1$, and $\theta= d(x,y)$.

We say $(X,d,\m)$ satisfies the \textit{reduced curvature-dimension condition} $CD^*(K,N)$ for $K\in \mathbb{R}$ and $N\in (0,\infty)$ if we replace the coefficients $\tau^{(t)}_{K,N}(\theta)$ with $\sigma_{K,N}^{(t)}(\theta)$.
%\begin{remark}
%If $(X,d,\m)$ is complete and satisfies the condition $CD(K,N)$ for $N<\infty$, then $(\supp \m, d)$ is a geodesic space and $(\supp\m,  d,\m)$ is 
%$CD(K,N)$. In the following we can always assume that $\supp\m=X$.
%\end{remark}
%\begin{remark} 
%
%If $K=0$, the condition $\CD(K,N)$ coincides with the condition $\CD^*(K,N)$ and is simply convexity of the $N$-Renyi entropy functional.
\end{definition}
\begin{remark}
By the inequality \eqref{ineq:dist} the condition $CD(K,N)$ always implies the condition $CD^*(K,N)$ and the latter is equivalent to a local version of $CD(K,N)$.
Under the assumptions that $(X,d,\m)$ is essentially nonbranching and $\m$ is finite Cavalletti and Milman \cite{cavmil} prove that $CD(K,N)$ and $CD^*(K,N)$ are equivalent (compare with Theorem \ref{thm:cavmil} below).
\end{remark}
\begin{definition}
A metric measure space $(X,d,\m)$ satisfies the Riemannian curvature-dimension condition $RCD(K,N)$ (or $RCD^*(K,N)$)  if it satisfies the condition $CD(K,N)$ (or $CD^*(K,N)$) and is infinitesimally Hilbertian, that is the corresponding Cheeger energy is quadratic.
\end{definition}
\begin{remark}
Since $RCD(K,N)$ and $RCD^*(K,N)$ spaces are essentially non-branching, the two conditions are equivalent provided $\m$ is finite (compare with Remark 2.15 in \cite{Kap-Ket-18}.
\end{remark}

\subsection{Alexandrov spaces}\label{intro:Alex}
In the following we introduce metric spaces with Alexandrov curvature bounded from below. For an introduction to this subject we refer to \cite{bbi}.
\begin{definition}
We define $\md_\kappa:[0,\infty)\rightarrow [0,\infty)$ as the solution of 
\begin{align*}
v''+ \kappa v=1 \ \ \ v(0)=0 \ \ \&\ \ v'(0)=0.
\end{align*}
More explicitly
\begin{align*}
\md_{\kappa}(x)=\begin{cases} \frac{1}{\kappa}\left(1-\cos_{\kappa}x\right)\ &\ \mbox{ if } \kappa\neq 0,\\
                 \frac{1}{2}x^2\ &\ \mbox{ if }\kappa =0.
                \end{cases}
                \end{align*}
\end{definition}
\begin{definition}
Let $(X,d)$ be a complete geodesic metric space. We say $(X,d)$  has curvature bounded  below by $\kappa\in\mathbb{R}$ in the sense of Alexandrov  if for any  unit speed geodesic $\gamma : [0,l]\to X$ 
such that 
\begin{equation}
d(y,\gamma(0))+l+d(\gamma(l),y)<2\pi_k,
\end{equation}
it holds that 
\begin{align}
\left[\md_{\kappa}(d_y\circ\gamma)\right]''+\md_{\kappa}(d_y\circ\gamma)\leq 1.
\end{align}
If $(X,d)$ has curvature bounded from below for some $k\in \R$ in the sense of Alexandrov, we say that $(X,d)$ is an Alexandrov space.
\end{definition}

\begin{remark} Alexandrov spaces are non-branching.
\end{remark}
\begin{theorem}[Petrunin, \cite{palvs}]\label{th:petrunincd}
Let $(X,d)$ be an $n$-dimensional Alexandrov space with curvature bounded from below by $k$. 
Then, $(X,d,\mathcal H^n_X)$ satisfies the condition $CD(k(n-1),n)$.
\end{theorem}
\begin{corollary}\label{cor:rcd}
Let $(X,d)$ be an $n$-dimensional Alexandrov space with curvature bounded from below by $\kappa$. Then, the metric measure space $(X,d,\mathcal H^n_X)$ satisfies the condition $RCD(\kappa(n-1),n)$.
\end{corollary}
\begin{proof}
The statement is known. Here, we give a straightforward argument for completeness.  It is enough to show that Cheeger energy is quadratic.

\begin{comment}
First we note that non-negatively curved Alexandrov spaces are closed w.r.t. Gromov-Hausdorff convergence, and  split off $\R$ provided there exists a geodesic line. Moreover, 
\end{comment}
It is known that 
a doubling condition and a 1-1 Poincar\'e inequality hold for Alexandrov spaces. 
Hence, we can follow the same argument as in \cite[Section 6]{Kap-Ket-18}. 
{It is known~\cite{BGP} that
for  $\mathcal H^n$-a.e. points $x\in X$ the tangent cone $T_pX$ is isometric to $\R^n$, and for a Lipschitz function the differential exists and is linear $\mathcal H^n$-a.e. \cite[Theorem 8.1]{cheegerlipschitz}. This implies the Cheeger energy is quadratic by the same argument as in \cite[Section 6]{Kap-Ket-18}.  }
\end{proof}
Let $(X,d)$ be an $n$-dimensional  Alexandrov space. We denote with $T_pX$ the  unique blow up tangent cone at $p\in X$. The tangent cone $T_pX$ coincides with the metric cone $C(\Sigma_p)$ where $\Sigma_p$ is the space of directions at $p$ equipped with the angle metric. The definition of the angle metric is as follows. The angle $\angle(\gamma^1,\gamma^2)$ between two geodesics $\gamma^i$, $i=1,2$, with $\gamma^1(0)=\gamma^2(0)=p$  and  parametrized by arclength is defined by the formula
\begin{align*}
\cos \angle(\gamma^1,\gamma^2)=\lim_{s,t\rightarrow 0} \frac{s^2+t^2-d(\gamma^1(s),\gamma^2(t))}{2st}.
\end{align*}
Then, the space of directions $\Sigma_pX$ is given as the metric completion of  $S_pX$ via $\angle$ where $S_pX$ is the space of geodesics starting in $p$. We refer to \cite{bbi} for details.
One can show that $(\Sigma_pX,\angle)$ is an $(n-1)$-dimensional Alexandrov space with curvature bounded below by $1$.
We say $p\in X$ is a regular point if $\Sigma_pX=\mathbb S^{n-1}$. We denote the set of regular points with $X^{reg}$.
As was mentioned above $\mathcal H^n$-almost every point $p\in X$ is regular. A theorem of Petrunin \cite{petpar} is the next statment. If $\gamma:[a,b]\rightarrow X$ is a geodesic such that there exists $t_0\in [a,b]$ with $\gamma(t_0)=X^{reg}$ then $\gamma(t_0)\in X^{reg}$ for all $t\in [a,b]$.

One can define the boundary $\partial X\subset X$ of $X$ via induction over the dimension.
One says that $p\in X$ is a boundary point if $\partial \Sigma_pX\neq \emptyset$. $\partial X$ denotes the set of all boundary points, and we call $\partial X$ the boundary of $X$.

Let $p\in X$ be a boundary point, that is $\partial \Sigma_p X\neq \emptyset$. We say $v\in \Sigma_p X$ is a normal vector in $p$ if $\angle(v,w)=\frac{\pi}{2}$ for any $w\in \partial \Sigma_p X$.

\begin{theorem}[Perelman \cite{Per-Morse},{ \cite[Lemma 4.3]{Per-Pet}} ]
For any point in an $n$-dimensional Alexandrove space there exists an arbitrary small, closed, geodesically convex neighborhood.
\end{theorem}

\subsection{Gluing}\label{subsec:gluing}
Let $(X_0,d_{X_0})$ and $(X_1,d_{X_1})$ be complete, $n$-dimensional Alexandrov spaces with non-empty boundaries $\partial X_0$ and $\partial X_1$ equipped with their intrinsic distances $d_{\partial X_0}$ and $d_{\partial X_1}$ respectively. Let $\mathcal I:\partial X_0\rightarrow \partial X_1$ be an isometry.

The topological glued space of $X_0$ and $X_1$ along their boundaries w.r.t. $\mathcal I$ is defined as the quotient space $X_0\dot \cup X_1/R$ of the disjoint union $X_0\dot \cup X_1$ where 
\begin{align*}
x\sim_ R y \mbox{ if and only if }\mathcal I(x)=y\mbox{ if }x\in \partial X_0, y\in \partial X_1, \mbox{ and }x=y\mbox{  otherwise.}
\end{align*}
The equivalence relaiton $R$ induces a pseudo distance on $X_0\dot\cup X_1=X$ as follows.  First, we introduce an extended metric $d$ on $X_0\dot \cup X_1$ via $d(x,y)=d_{X_i}(x,y)$ if $x,y\in X_i$ for some $i\in \{0,1\}$ and $d(x,y)=\infty$ otherwise.
Then, for $x,y\in X_0\dot \cup X_1$ we define
{
\begin{align*}
\hat d(x,y) = \inf  \sum_{i=0}^{k-1} d(p_{i},q_i)
\end{align*}
where the infimum runs over all collection of tuples 
 $\{(p_i,q_{i})\}_{i=0,\dots ,k-1}\subset X \times X$ for some $k\in N$
 such that $q_i\sim_R p_{i+1}$, for all $i=0,\dots, k-1$ and $x=p_0, y=q_k$.} One can show that $x\sim_R y$ if and only if $\hat d(x,y)=0$ if $X_0$ and $X_1$ are Alexandrov spaces.
The glued space between $(X_0,d_{X_0})$ and $(X_1,d_{X_1})$ w.r.t. $\mathcal I:\partial X_0 \rightarrow \partial X_1$ is the metric space defined as 
$$X_0\cup_{\mathcal I} X_1:=(X_0\cup X_1/R, \hat{d}).$$
In the following we denote the glued space as $(Z,d_Z)$, and boundary $\partial X_0$ with its intrinsic metric with $(Y,d_Y)$.
In the case when $X_0=X_1=X$ and $\mathcal I=\mbox{id}_{\partial X}$, we call $X\cup_{\mathcal I}X=:\hat X$ the double space of $X$.
\begin{remark}\label{rem:local}  For every point $p\in X_i\backslash Y$, $i=0,1$, there exists $\epsilon>0$ such that $B_{\epsilon}(p)\subset X_i$ and $d_Z|_{B_\epsilon(p)\times B_\epsilon(p)}= d_{X_i}|_{B_\epsilon(p)\times B_\epsilon(p)}$.
\end{remark}
\begin{theorem}[Petrunin, \cite{pg}]\label{th:petruningluing}
Let $(X_0,d_{X_0})$ and $(X_1,d_{X_0})$ be $n$-dimensional Alexandrov spaces with nonempty boundary and curvature bounded from below by $k$. Let $\mathcal I:\partial X_0\rightarrow \partial X_1$ be an isometry w.r.t. the induces intrinsic metrics. 
Then, $X_0\cup_{\mathcal I} X_1$ is an Alexandrov space with curvature bounded from below by $k$.
\end{theorem}
\begin{remark}
The special case of a double space was proven first by Perelman \cite{p}.
\end{remark}
\begin{remark}
By symmetry of the construction one can see that geodesics in the double space $\hat X$ of an Alexandrov space $X$ { connecting points in $\hat X\backslash \partial X$} intersect with the boundary at most once, and the restriction of the double metric to $X\backslash \partial X$ coincides with $d_X$. This observation was crucial in Perelman's proof of the double theorem. 
{However, in the general case of glued spaces it's not clear if geodesics connecting points in $Z\backslash Y$ intersect $Y$ at most finitely many times. This creates an extra difficulty in the proof of Petrunin's theorem and also in the proof of Theorem~\ref{main}. }
\end{remark}

Let us recall some additional facts about the glued space $Z$ \cite{pg}. Since the boundary $Y\subset X_0$ is an extremal subset in $X_0$, the following holds. Consider the blow up tangent cone $\lim_{\epsilon\rightarrow 0} (X_0, \frac{1}{\epsilon} d_{X_0}, p) =T_p X_0$ for $p\in Y$. Then, $\lim_{\epsilon\rightarrow 0}(Y, \frac{1}{\epsilon}d_Y, p)= T_p Y$ w.r.t. the intrinsic metric $d_Y$ on $Y$ is equal to $C(\partial \Sigma_p X_0)= \partial C(\Sigma_p X_0)$.

It follows that $\partial \Sigma_p X_0$  is isometric to $\partial \Sigma_p X_1$ via an isometry $\mathcal I'$ that arises as blow up limit of $\mathcal I$. 

Then it also follows from Petrunin's proof of the glued space theorem that $ T_p Z= T_p X_0\cup_{\mathcal I'}T_pX_1$ and  $\Sigma_p Z= \Sigma_pX_0 \cup_{\mathcal I'} \Sigma_p X_1$.

If $p\in Y$ is a regular point in the glued space $Z$, that is $\Sigma_pZ=\mathbb S^{n-1}$, it follows by maximality of the volume of $\mathbb S^{n-1}$ in the class of Alexandrov spaces with curvature bounded below by $1$ that $\Sigma_p X_0= \mathbb{S}^{n-1}_+$ and $\Sigma_pX_0=\mathbb S^{n-1}_-$ where $\mathbb{S}^{n-1}_{+/-}$ denote the lower and upper half sphere respectively, and $\Sigma_p Y= \partial \Sigma_p X_0=\mathbb S^{n-2}$. In particular, the north pole $N$ in $\mathbb S^{n-1}_+$ is the unique normal vector a $p\in Y\subset X_0$, and the south pole $S\in \mathbb{S}^{n-1}_-$ is the unique normal vector a $p\in Y\subset X_1$.
\subsection{Semi-concave functions}
\noindent
We recall a few basic facts about concave functions following \cite{plaut}.

A function
$u:[a,b]\rightarrow \R$ is called concave if the segment between any pair of points lies below the graph. If $u$ is concave, $u$ is lower semi continuous and continuous on $(a,b)$. The secant slope
$
\frac{u(s)-u(t)}{s-t}
$
is a decreasing function in $s$ and $t$. It follows that the right and left derivative
\begin{align*}
\frac{d^+}{dr} u (r)=\lim_{h\downarrow 0} \frac{u(r+h)-u(r)}{h}
 \ \ \&\ \ 
\frac{d^-}{dr} u (r)=\lim_{h\downarrow 0} \frac{u(r-h)-u(r)}{-h}
\end{align*}
exist in $\R\cup\{\infty\}$ and $\R\cup \{-\infty\}$ respectively
for all $r\in [a,b]$ with values in $\R$ if $r\in (a,b)$. Moreover $\frac{d^+}{dr}u(r)\leq \frac{d^-}{dr}u(r)$ and $\frac{d^{+/-}}{dr}u(r)$ are decreasing in $r$. 
If $\frac{d^+}{dr}u(a)<\infty$ ($\frac{d^-}{dr}u(b)>-\infty$), $u$ is continuous in $a$ (in $b$).

Let $u: [a,b]\rightarrow (0,\infty)$ satisfy
\begin{align}\label{kuconcavity1}
u\circ\gamma(t)\geq \sigma_{\kappa}^{(1-t)}(|\dot{\gamma}|)u\circ \gamma(0) + \sigma_{\kappa}^{(t)}(|\dot{\gamma}|)u\circ\gamma(1)
\end{align}
for any constant speed geodesic $\gamma:[0,1]\rightarrow [a,b]$. 
It follows that $u$ is lower semi continuous and continuous on $(a,b)$.

\begin{definition}
Let $f: [a,b]\rightarrow \R$ be continuous on $(a,b)$, and let $F:[a,b]\rightarrow \R$ such that
$F''=f$ on $(a,b)$.
For a function $u: [a,b]\rightarrow \R$ 
we write $u''\leq f$ on $(a,b)$ if $u-F$ is concave on $(a,b)$.
\end{definition}
We say a function $u:(0,\theta)\rightarrow \R$ is $\lambda$-concave if $u''\leq \lambda$. We say $u$ is semiconcave if for any $r\in (0,\theta)$ we can find $\epsilon>0$ and $\lambda\in \R$ such that $u$ is $\lambda$-concave on $(r-\epsilon,r+\epsilon)$.

If $u$ satisfies \eqref{kuconcavity1} for every constant speed geodesic $\gamma:[0,1]\rightarrow [a,b]$, then one can check that
\begin{align*}
u''+ku\leq 0\mbox{ on }(a,b)
\end{align*}
in the sense of the previous definition. We note that \eqref{kuconcavity1} implies that $u$ is continuous on $(a,b)$, and 
$U(t)=\int_a^bg(s,t) u(s) ds$ satisfies $U''=-u$ on $(a,b)$ where $g(s,t)$ is the Green function of the interval $[a,b]$.

On the other hand we have the next lemma. 
\begin{lemma}
Let $u:[a,b]\rightarrow \R$ be lower semi-continuous and continuous on $(a,b)$ such that $u''+ku\leq 0$ on $(a,b)$ in the sense of the definition above.

Then $u$ satisfies \eqref{kuconcavity1} for every constant speed geodesic $\gamma:[0,1]\rightarrow [a,b]$.
\end{lemma}
\begin{proof}
We sketch the proof. If $u''+k u\leq 0$ then $u-k U$  is concave. In particular, it follows for $\phi\in C^{2}_c((a,b))$, $\phi\geq 0$, that
\begin{align*}
0\geq \int (u+kU)\phi ''dt = \int u \phi'' + k\int u\phi dt
\end{align*}
by the distributional characterisation of convexity (see \cite{simonconvexity}).
Hence, $u$ satisfies $u''+k u\leq 0$ in distributional sense, and therefore \eqref{kuconcavity1} follows by \cite[Lemma 2.8]{erbarkuwadasturm}. 
\end{proof}
\begin{lemma}\label{lem:local}
If $u$ satisfies \eqref{kuconcavity1} for every constant speed geodesic $\gamma:[0,1]\rightarrow [a,b]$ of length less than $\theta<b-a$, then $u$ satisfies \eqref{kuconcavity1}.
\end{lemma}
If $u:[a,b]\rightarrow \R$ satifies \eqref{kuconcavity1}, it is is semi-concave, therefore locally Lipschitz on $(a,b)$ and hence differentiable $\mathcal{L}^1$-almost everywhere.
Moreover, the right and left derivative also exist in this case and satisfy $\frac{d^+}{dr}u(r)\leq \frac{d^-}{dr}u(r)$ with equality if and only if $u$ is differentiable in $r$.  $\frac{d^{+/-}}{dr}u$ is continuous from the right/left.
Since $u$ is locally semi concave, the second derivative $u''$ exists $\mathcal L^1$-almost everywhere.

The following Lemma can be found in \cite{plaut} (Lemma 113). 
\begin{lemma}\label{importantlemma}
Consider $u: (a,b)\rightarrow \R$ continuous such that $u''\leq - ku$ on $(a,c)$ and on $(c,b)$ for some $c\in (a,b)$.
Then $u''\leq -k u$ on $(a,b)$ if and only if 
\begin{align*}
\frac{d^-}{dr} u (c)\geq \frac{d^+}{dr} u (c).
\end{align*}
\end{lemma}
\begin{definition}\label{def:semiconcavefct}
Let $(X,d)$ be an $n$-dimensional Alexandrov space and $\Omega\subset X$. A function $f:\Omega\rightarrow \R$ is $\lambda$-concave if $f$ is {\bf locally Lipschitz} and $f\circ\gamma:[0,\mbox{L}(\gamma)]\rightarrow \R$ is $\lambda$-concave for every constant speed geodesic $\gamma:[0, \mbox{L}(\gamma)]\rightarrow \Omega$. A function $f:X\rightarrow \R$ is semi-concave if for every $p\in X$ there exists a neighborhood $U\ni p$ such that $f|_U$ is { $\lambda$-concave for some real $\lambda$.}

We say a function $f: X\rightarrow \R$ is double semi-concave if the function $f\circ P:\hat X\rightarrow \R$ is semi-concave
where the $P: \hat X \rightarrow X$ is the projection map form the double space $\hat X$ to $X$. If $\partial X=\emptyset$, concavity and double concavity coincide. 

In \cite{petsem} Petrunin defines concavity as double concavity. 

\end{definition}\label{def:semiconcavefcts}
Let $X$ be an Alexandrov space and let $f: X \rightarrow \R$ be locally Lipschitz. Then, the limit
\begin{align*}
\lim_{r\downarrow 0} \frac{f\circ \gamma(r) - f\circ \gamma(0)}{r}=\frac{d^+}{dr}( f\circ \gamma)(0)=: df_p(\dot{\gamma})=:df(\dot{\gamma}) \in \R
\end{align*}
exists for every geodesic $\gamma:[0,\theta]\rightarrow X$ parametrized by arc length with $\gamma(0)=p$, and for every $p\in X$.
We call $df: T_pX\rightarrow \R$ the differential of $f$. 

The differential $df_p$ on $T_pX$ can be equivalently  defined as limit of the sequence $\frac{1}{\epsilon} \left(f-f(p)\right):( \frac{1}{\epsilon}X,p) \rightarrow \R$. This limit is understood in the sense of Gromov's Arzela-Ascoli theorem (see for instance \cite{sf}.  { It also makes  sense for functions that are just locally Lipschitz but the differential is not unique in this case}. {Note that since Alexandrov spaces are nonbranching,  under GH convergence of Alexandrov spaces every geodesic in the limit is a limit of geodesics in the sequence. Therefore it follows that  $df_p:T_pX\rightarrow \R$ is Lipschitz for Lipschitz functions, and also concave if $f$ is semiconcave. This in turn implies that $v=df_p:\Sigma_p\rightarrow \R$ satisfies $v''+v\leq 0$ along geodesics in $\Sigma_p$.}

\subsection{$1D$ localisation of generalized Ricci curvature bounds.}\label{subsec:1Dlocalisation}
\noindent
In this section we will recall the localisation technique introduced by Cavalletti and Mondino. The presentation follows Section 3 and 4 in \cite{cavmon}. We assume familarity with basic concepts in optimal transport.

Let $(X,d,\m)$ be a locally compact metric measure space that is essentially nonbranching. We  assume that $\supp\m =X$.

Let $u:X\rightarrow \mathbb{R}$ be a $1$-Lipschitz function. Then 
\begin{align*}
\Gamma_u:=\{(x,y)\in X\times X : u(x)-u(y)=d(x,y)\}
\end{align*}
is a  $d$-cyclically monotone set, and one defines $\Gamma_u^{-1}=\{(x,y)\in X\times X: (y,x)\in \Gamma_u\}$. 
If $\gamma\in \mathcal G^{[a,b]}(X)$ for some $[a,b]\subset\R$ such that $(\gamma(a),\gamma(b))\in \Gamma_u$ then $(\gamma(t),\gamma(t))\in \Gamma_u$ for $a<t\leq s<b$.  It is therefore natural to consider the set $G$ of unit speed transport geodesics $\gamma:[a,b]\rightarrow \R$ such that $(\gamma(t),\gamma(s))\in \Gamma_u$ for $a\leq t\leq s\leq b$. 

The union $\Gamma\cup \Gamma^{-1}$ defines a relation $R_u$ on $X\times X$, and $R_u$ induces a {\it transport set with endpoints} 
$$\mathcal T_u:= P_1(R_u\backslash \{(x,y):x=y\})\subset X$$ 
where $P_1(x,y)=x$. For $x\in \T_u$ one defines $\Gamma_u(x):=\{y\in X:(x,y)\in \Gamma_u\}, $
and similar $\Gamma_u^{-1}(x)$ as well as $R_u(x)=\Gamma_u(x)\cup \Gamma_u^{-1}(x)$. Since $u$ is $1$-Lipschitz, 
 $\Gamma_u, \Gamma_u^{-1}$ and $R_u$ are closed as well as $\Gamma_u(x), \Gamma_u^{-1}(x)$ and $R_u(x)$.

The {\it transport set without branching} $\mathcal T^b_u$ associated to $u$ is then defined as 
\begin{align*}
\mathcal T^b_u=\left\{ x\in \mathcal T_u: \forall  y,z\in R_u(x) \Rightarrow  (y,z)\in R_u\right\}
\end{align*}
$\T_u$ and $\T_u\backslash \T^b_u$ are $\sigma$-compact, and $\T^b_u$ and $R_u\cap \T^b_u\times \T^b_u$ are Borel sets.
In \cite{cavom} Cavalletti shows that $R_u$ restricted to $\T^b_u\times \T^b_u$ is an equivalence relation. 
Hence, from $R_u$ one obtains a partition of $\mathcal T^b_u$ into a disjoint family of equivalence classes $\{X_{\gamma}\}_{\gamma\in Q}$. Moreover, $\T^b_u$ is also $\sigma$-compact. 

Every $X_{\gamma}$ is isometric to some interval  $I_\gamma\subset\mathbb{R}$ via an isometry $\gamma:I_\gamma \rightarrow X_{\gamma}$. $\gamma:I_\gamma\rightarrow X$ extends to a geodesic that is arclength parametrized and that we also denote $\gamma$ defined on the closure $\overline{I}_{\gamma}$ of $I_{\gamma}$. We set $\overline{I}_{\gamma}=[a_\gamma,b_\gamma]$.

The set of equivalence classes $Q$ has a measurable structure such that $\mathfrak Q: \T^b_u\rightarrow Q$ is a measurable map. 
We set $\mathfrak q:= \mathfrak Q_{\#}\m|_{\mathcal T^b_u}$. 

Recall that a measurable section of the equivalence relation $R$ on $\T^b_u$ is a measurable map $s: \T^b_u\rightarrow \T^b_u$ such that $R_u(s(x))=R_u(x)$ and $(x,y)\in R_u$ implies $s(x)=s(y)$. In \cite[Proposition 5.2]{cavom} Cavalletti shows there exists a measurable section $s$ of $R$ on $\T^b_u$. Therefore, one can identify the measurable space $Q$ with $\{x\in \T^b_u: x=s(x)\}$ equipped with the induced measurable structure and we can see $\mathfrak q$ as a Borel measure on $X$. By inner regularity there exists a $\sigma$-compact set $Q'\subset X$ such that $\mathfrak q(Q\backslash Q')=0$ and in the following we will replace $Q$ with $Q'$ without further notice. {We parametrize $\gamma\in Q$ such that $\gamma(0)=s(x)$. In particular, $0\in (a_\gamma, b_\gamma)$. }

Now, we assume that $(X,d,\m)$ is an essentially non-branching $CD^*(K,N)$ space for $K\in \R$ and $N\geq 1$. 
The following lemma is Theorem 3.4 in \cite{cavmon}.
\begin{lemma}\label{somelemma}
Let $(X,d,\m)$ be an essentially non-branching $CD^*(K,N)$ space for $K\in \R$ and $N\in (1,\infty)$ with $\supp \m=X$ and $\m(X)<\infty$.
Then, for any $1$-Lipschitz function $u:X\rightarrow \R$, it holds $\m(\T_u\backslash \T^b_u)=0$.
\end{lemma}
For $\mathfrak q$-a.e. $\gamma\in Q$ it was proved in \cite{cavmil} (Theorem 7.10) that 
\begin{align*}
R_u(x)=\overline{X_\gamma}\supset X_\gamma \supset (R_u(x))^{\circ} \ \ \forall x\in \mathfrak Q^{-1}(\gamma).
\end{align*}
where $(R_u(x))^\circ$ denotes the relative interiour of the closed set $R_u(x)$. 
%Hence, the only points of $R_u(x)$ that are eventually not contained in $\T^b_u$ are the endpoints of $\gamma:[a_\gamma,b_\gamma]\rightarrow X$ where $\mathfrak Q(x)=\alpha$. Therefore $\overline{X}_{\gamma}=\mbox{Im}(\gamma)$.
\begin{theorem}
Let $(X,d,\m)$ be a compact geodesic metric measure space with $\supp\m =X$ and $\m$ finite. Let $u:X\rightarrow \mathbb{R}$ be a $1$-Lipschitz function,  let $(X_{\gamma})_{\gamma\in Q}$ be the induced partition of $\mathcal T^b_u$ via $R_u$, and let $\mathfrak Q: \T^b_u\rightarrow Q$ be the induced quotient map as above.
Then, there exists a unique strongly consistent disintegration $\{\m_{\gamma}\}_{\gamma\in Q}$ of $\m|_{\T^b_u}$ w.r.t. $\mathfrak Q$. 
\end{theorem}
Define the ray map 
\begin{align*}
g:  \mathcal V\subset Q\times \R\rightarrow X\ \mbox{ via }\
\mbox{graph}(g)=\{ (\gamma, t,x) \in Q\times \R\times X:  \gamma(t)=x\}
\end{align*}
By definition $\mathcal V= g^{-1}(\T^b_u)$. The map $g$ is Borel measurable, $g(\gamma,\cdot)=\gamma: (a_\gamma, b_\gamma)\rightarrow X$ is a geodesic, $g:\mathcal V\rightarrow  \T^b_u$ is bijective and its inverse is given by $g^{-1}(x)=( \mathfrak Q(x), \pm d(x,\mathfrak Q(x)))$.
\begin{theorem}\label{th:1dlocalisation}
Let $(X,d,\m)$ be an essentially non-branching $CD^*(K,N)$  space with $\supp\m=X$, $\m(X)<\infty$, $K\in \R$ and $N\in (1,\infty)$.

Then, for any $1$-Lipschitz function $u:X\rightarrow \R$ there exists a disintegration $\{\m_{\gamma}\}_{\gamma\in Q}$ of $\m$ that is strongly consistent with $R^b_u$. 

Moreover, for $\mathfrak q$-a.e. $\gamma\in Q$, $\m_{\gamma}$ is a Radon measure with $\m_{\gamma}=h_{\alpha}\mathcal{H}^1|_{X_{\alpha}}$ and $(X_{\gamma}, d_{X_\gamma}, \m_{\gamma})$ verifies the condition $CD(K,N)$.

More precisely, for $\mathfrak q$-a.e. $\gamma\in Q$ it holds that
\begin{align}\label{kuconcave}
h_{\gamma}(c_t)^{\frac{1}{N-1}}\geq \sigma_{K/N-1}^{(1-t)}(|\dot c|)h_{\gamma}(\gamma_0)^{\frac{1}{N-1}}+\sigma_{K/N-1}^{(t)}(|\dot c|)h_{\gamma}(\gamma_1)^{\frac{1}{N-1}}
\end{align}
for every geodesic $c:[0,1]\rightarrow (a_\gamma,b_\gamma)$.
\end{theorem}
\begin{remark} The property
\eqref{kuconcave} yields that $h_{\gamma}$ is locally Lipschitz continuous on $(a_\gamma,b_\gamma)$ \cite[Section 4]{cavmon}, and that $h_{\gamma}:\R \rightarrow (0,\infty)$ satifies
\begin{align*}
\frac{d^2}{dr^2}h_{\gamma}^{\frac{1}{N-1}}+ \frac{K}{N-1}h_{\gamma}^{\frac{1}{N-1}}\leq 0 \mbox{ on $(a_\gamma, b_\gamma)$} \mbox{ in distributional sense.}
\end{align*}
\end{remark}
\subsection{Characterization of curvature bounds via $1D$ localisation}
\begin{definition}
Let $(X,d_X,\m_X)$ be an essentially non-braching metric measure space with $\m(X)=1$, let $K\in \R$ and $N\geq 1$, and let $u:X\rightarrow \R$ be a $1$-Lipschitz function. We say that $(X,d_X,\m_X)$ satisfies the condition $CD_u^1(K,N)$ if there exist subsets $X_\gamma\subset X$, $\gamma\in Q$, such that
\begin{itemize}
\item[(i)] There exists a disintegration of $\m_{\mathcal T_u}$ on $(X_\gamma)_{\gamma\in Q}$: 
\begin{align*}
\m|_{\mathcal T_u}=\int_{X_\gamma} \m_\gamma d\mathfrak q(\gamma)\ \mbox{ with }\m_\gamma(X_\gamma)=1\ \mbox{for $\mathfrak q$-a.e. }\gamma\in Q.
\end{align*}
\item[(ii)] For $\mathfrak q$-a.e. $\gamma\in Q$ the set $X_\gamma$ is the image $\mbox{Im}(\gamma)$ of a geodesic $\gamma:I_\gamma\rightarrow X$ for an interval $I_\gamma\subset \R$.
\item[(iii)] The metric measure space $(X_\gamma, d_{X_\gamma},\m_\gamma)$ satisfies the condition $CD(K,N)$.
\end{itemize}

The metric measure space $(X,d_X,\m_X)$ satisfies the condition $CD^1_{Lip}(K,N)$ if it satisfies the condition $CD^1_u(K,N)$ for any $1$-Lipschitz function $u:X\rightarrow \R$.
\end{definition}
\begin{remark}
From the previous subsection it is immediatly clear that the condition $CD(K,N)$ implies the condition $CD^1_{Lip}(K,N)$. 
\end{remark}
{
The following theorem will play an important role in the proof of our gluing result.}
\begin{theorem}[Cavalletti-Milman]\label{thm:cavmil}
If an essentially non-branching metric measure space $(X,d_X,\m_X)$ satisfies the condition $CD_{Lip}^1(K,N)$ for $K\in \R$ and $N\in [1,\infty)$ then it satisfies the condition $CD^*(K,N)$.

If $\m$ is a finite measure it even satisfies $CD(K,N)$.
\end{theorem}
{
\begin{remark}
Taking into account a disintegration result for $\sigma$-finite measures in \cite[Section 3.1]{cav-mon-lapl-18} it { should be} possible to prove the previous theorem also for $\sigma$-finite measures. Consequently in our main theorem we {would then be able}  to replace the condition $CD^*$ with the condition $CD$ in general.
\end{remark}}

\section{Applying $1D$ localisation}

\subsection{First application}\label{sec:first application}
Let $X_0$ and $X_1$ be $n$-dimensional Alexandrov spaces, let $\mathcal I:\partial X_0\rightarrow \partial X_1$ be an isometry as in Theorem \ref{th:petruningluing}. Let $Z$ be the glued space, set $\partial X_0 =\partial X_1=:Y$ and recall that $(Z,d_Z,\mathcal H^n)$ satisfies $CD(k(n-1),n)$ for some $k\in \R$. We consider continuous functions $\Phi_0$ and $\Phi_1$ on $X_0$ and $X_1$ respectively such that $\Phi_0|_{\partial X_0}=\Phi_1|_{\partial X_1}$, 
and we define $\Phi_Z: X_0\cup_{\mathcal I} X_1\rightarrow \R$ by 
\begin{align*}
\Phi_Z(x)=\begin{cases} \Phi_0(x)\ & \ \mbox{if } x\in X_0, \\
\Phi_1(x)\ & \ \mbox{otherwise}.
\end{cases}
\end{align*}
{

\begin{lemma}\label{lemma} Let $X$ be an $n$-dimensional Alexandrov space with $Y=\partial X\neq \emptyset$.
Let $\Phi:X\rightarrow \R$, $i=0,1$ be {semi-concave}. 
Then $\Phi|_{Y}:Y\rightarrow \R$ is differentiable $\mathcal H^{n-1}$-a.e. {meaning that $Y$ contains a subset $A$ such that $\mathcal H^{n-1}(Y\backslash A)=0$ and for every $a\in A$ it holds that $T_aY\cong \R^{n-1}$ and $d\Phi\co T_aY\to \R$ is linear.}
\end{lemma}
\begin{proof}
% 
% More precisely, one can cover $Y$ with the image of bi-Lipschitz maps $\phi_i: B_i\rightarrow Y$, $i\in\N$, for Borel sets $B_i\subset \R^{n-1}$ up to a set of $\mathcal H^n$-measure $0$. Moreover $\Phi|_Y$ is Lipschitz. Hence, by the metric version of Rademacher's theorem \cite{cheegerlipschitz} it follows that $\Phi|_Y$ is differentiable $\mathcal H^{n-1}$-a.e.\ .
% % 
%  Alternatively, this follows from the standard Rademacher theorem on $\R^{n-1}$ in combination with Lytchak's metric differentiability theorem which guaranties that a Lipschitz map from an open subset in a  Eucliudean space into an Alexandrov space is a.e. differentiable \cite[Corollary 1.5]{Lytchak04}. Applied to a biLipschitz local distance  coordinate map $F \co U\to V\subset \R^{n-1}$ on an open neighborhood $U$ of a regular  point $p\in Y$ we get that both $\Phi\circ F^{-1}$ and $F^{-1}$ are differentiable a.e. on $V$ and hence $F$ is differentiable a.e. on $U$. Since the set of regular points in $Y$ has full measure the claim follows. 

{

One says a point $p\in Y=\partial X $ is boundary regular if $T_pY=\R^{n-1}$.
Since $Y$ is the boundary of  an $n$-dimensional Alexandrov space $X$, it is $\mathcal H^{n-1}$-rectifiable. Even stronger, 
the set of boundary regular points $\mathcal R(Y)$ in $Y$ has full $\mathcal H_Y^{n-1}$-measure and for any $\epsilon>0$ one can cover  $\mathcal R(Y)$ 
by  $(1+\epsilon)$-biLipschitz coordinate maps $F^{\epsilon}_i:U_i\to V_i\subset \R^{n-1}$ where each $V_i$ is open. The maps $F^{\epsilon}_i$ are given by standard strainer coordinates centered at boundary regular  points.
By the metric version of Rademacher's theorem due to Cheeger \cite{cheegerlipschitz} it follows that $\Phi|_Y$ is differentiable $\mathcal H^{n-1}$-a.e.\ .

Let us give another, self-contained argument that does not rely on Cheeger's theorem.

Without loss of generality we can assume that $\Phi$ is $L$-Lipscjhitz for some finite $L>0$.
For $0<\epsilon<1$ we consider the maps $F^\epsilon_i$.
Let us drop the superscript $\epsilon$ for a moment.  In particular, each coordinate component $F_i^j$, $j=1,\dots, n-1$, of $F_i$ is a semiconcave function on $U_i$ and admits a differential $dF^j_i|_p$ in the sense of Alexandrov spaces at every point $p\in U_i$. 
Since $F_i$ and $F_i^{-1}$ are $(1+\epsilon)$-biLipschitz, one has that $(1+\epsilon)|v|\leq |dF^j_i|_p(v)|\leq (1+\epsilon)|v|$ for every $p\in U_i$ and every $v\in T_pX$. 

{
 The function $\Phi\circ F_i^{-1}: V_i\rightarrow \R$ is $2L$-Lipschitz and therefore differentiable $\mathcal L^{n-1}$-a.e. by the standard Rademacher theorem.
So we can choose a set of full measure $W_i^{\epsilon}=W_i$ in $U_i$ such that $\forall p\in F_i(W_i)\subset V_i$ the point $p\in W_i$ is regular and the function $\Phi\circ F_i^{-1}$ is differentiable at $F_i(p)$.

The chain rule for Alexandrov space differentials yields
\begin{align}\label{someid}
d\Phi|_p = d (\Phi \circ F_i^{-1})|_{F_i(p)}\circ DF_i|_p \ \ \forall p\in W_i
\end{align}
where $DF_i|_p= (dF_i^1|_p, \dots, dF_i^{n-1}|_p)$. }
%If $x\in V_j$ is a point of differentiability, the standard differential and the differential in Alexandrov sense coincide. 

We obtain by \eqref{someid} that  for all  $p\in W_i$ and for any $\epsilon>0$ the Alexandrov differential $d\Phi|_p: \R^{n-1}\mapsto \R$ of $\Phi$ at $p$ is the composition of a $2L$-Lipschitz  linear map $A^{\epsilon}=d (\Phi \circ F_i^{-1})|_{F_i(p)}:\R^{n-1}\mapsto \R$ and  1-homogeneous map $B^{\epsilon}=DF_i|_p:\R^{n-1}\mapsto \R^{n-1}$ that is $\epsilon$-close to an isometry on a the unit ball around the origin (we have identified $T_pY$ with $\R^{n-1}$ in the above).

Let us consider $\epsilon_n=\frac{1}{n}$ and let $p\in \bigcap_{n\in \N} \bigcup W_i^{\frac{1}{n}}$.  After eventually choosing a subsequence $A^{\frac{1}{n}} \rightarrow A$ for a linear map $A$ and $B^{\frac{1}{n}}\rightarrow B$ for an isometry $B$. Hence $d\Phi_p=A\circ B$ is linear. Since $\bigcap_{n\in \N} \bigcup W_i^{\frac{1}{n}}$ has full $\mathcal H^{n-1}$-measure this yields the claim.}

\end{proof}

\begin{proposition}\label{prop}

Let $Z$ be the glued space, let $Y=\partial X_i\subset Z$. Let $N\subset Y$ such that $\mathcal{H}^{n-1}(N)=0$. Let $(X_\gamma)_{\gamma\in Q}$ be the $1D$ localisation of $m=\mathcal H^n_Z$  w.r.t. the $1$-Lipschitz function $u=d(x_1,\cdot)$ for $x_1\in B_\eta(x_1)\subset X_1$ and $\eta>0$. Let $(Q,\mathfrak q)$ and $\mathfrak Q: \T^b_u \rightarrow Q$ be the corresponding quotient space and the quotient map.
Then
\begin{align*}
\mathcal H^n(X_0\cap \mathfrak Q^{-1}(\{\gamma\in Q: \exists t\in [0,L(\gamma)) \mbox{ s.t. }\gamma(t)\in N\}))=0.
\end{align*}
\end{proposition}
\begin{proof}
The property that $\mathcal H^{n-1}(N)=0$ is equivalent to the following statement. For any $\epsilon>0$ and any $\delta\in (0,\eta/4)$ there exist $(r_i)_{i\in \N}$ with $r_i\in (0,\delta)$ and $x_i\in Z$, $i\in \N$, such that
\begin{align}\label{nullset_estimate}
N\subset \bigcup_{i\in\N} B_{r_i}(x_i)\ \mbox{ and }\ 
\sum_{i\in \N}(r_i)^{n-1}\leq \epsilon.
\end{align}
We set $Q_i'=\{\gamma\in \mathcal G(X): \exists t\in [0,1] \mbox{ s.t. }\gamma(t)\in B_r(x_i)\}$ and $Q_i=Q_i'\backslash \bigcup_{j=1}^{i-1} Q_j$. W.l.o.g. we assume that $\mathfrak q(Q_i)>0$ for any $i\in \N$.
We will prove that 
\begin{align*}
\mathcal H^n\left(X_0\cap \bigcup_{i\in \N}\mathfrak Q^{-1}(Q_i)\right)\leq C \epsilon
\end{align*}
for some constant $C=C(k,n)$. This implies the claim of the proposition.
\smallskip

Let us fix $i\in \N$. 
There exists $Q^\dagger_i$ with $\mathfrak q(Q_i)=\mathfrak q(Q^\dagger_i)$ such that $\m_\gamma$ admits a density $h_\gamma$ w.r.t. $\mathcal H^1$ and $(X_\gamma, \m_\gamma)$ is $CD(k(n-1),n)$ for all $\gamma\in Q^\dagger_i$. In particular, if $J\subset I_\gamma$ and $J_t= t b_\gamma + (1-t) J$, then the Brunn-Minkowski inequality implies
\begin{align*}
\m_\gamma(J_t)\geq {C(k,n)}{t^n} \m_\gamma(J).
\end{align*}
We pick $J=\gamma^{-1}(X_0)$. 
Let $D=\diam Z$ and choose $s\in \N$ such that $\frac{D}{s-1}\leq r_i\leq \frac{D}{s}$. We decompose $J$ into intervals $(J^{l,s})_{l=1,\dots,s}$ such that $|J^{l,s}|\leq \frac{D}{s}$. Moreover, there exists $t_{l}\in (0,1)$ such that $J^{l,s}_{t_l}\subset \gamma^{-1}(B_{2r_i}(x_i))$. 
Since $r_i\leq \delta < \eta/4$ and since $B_\eta(x_1)\subset X_1$, we have $B_{\eta/2}(x_1)\cap B_{2r_i}(x_i)=\emptyset$. Therefore $t_l\geq \frac{\eta}{2 D}$. Hence
\begin{align*}
 \frac{1}{r_i}\m_\gamma(B_{2r_i}(x_i))\geq \frac{s}{D}\m_\gamma(J^{l,s}_{t_l})\geq \sum_{l=1}^s{C(k,n)}{t_l^n} \m_\gamma(J^{l,s})\geq {C(k,n,D)}{\eta^n} \m_\gamma(J).
\end{align*}
Integration w.r.t. $\mathfrak q$ on $Q_i$ yields
\begin{align*}
\hat C(k,n) r_i^{n-1}\geq \frac{1}{r_i} \mathcal H^n(B_{2r_i}(x_i))\geq {C(k,n,D,\eta)} \mathcal H^n(X_0\cap \mathfrak Q^{-1}(Q_i)).
\end{align*}
After summing up w.r.t. $i\in \N$ together with \eqref{nullset_estimate} and since $\{Q_i\}_{i\in \N}$ are disjoint, we obtain the claim and we proved the proposition.
\end{proof}
\subsection{Second application}\label{sec: Second application}
Let $u: X\rightarrow \R$ be a $1$-Lipschitz function, let $(\m_{\gamma})_{\gamma\in Q}$ be the induced disintegration {of $\mathcal H^n$}.
We pick a subset $\hat Q$ of full $\mathfrak q$ measure in $Q$ such that $R_u(x)=\overline{X_\gamma}$ for all $x\in X_\gamma$. By abuse of notation we write $\hat Q=Q$ and $\T_u=\mathfrak Q^{-1}(\hat Q)$. 

We say that  a unit speed geodesic $\gamma:[a,b]\rightarrow X$ is tangent to $Y$ if
 there exists $t_0\in [a,b]$ such that $\gamma(t_0)\in Y$ and $\dot{\gamma}(t_0)\in T_pY$.
We define 
\begin{align*}
Q^{\dagger}:= \left\{\gamma\in Q: \# \gamma^{-1}(Y)<\infty\right\}.
\end{align*}
\begin{lemma}\label{lem: infinite implies tangent}
If $\gamma\in Q\backslash Q^{\dagger}$, then $\gamma$ is tangent to $Y$.
\end{lemma}
\begin{proof}
If $\gamma\in Q\backslash Q^{\dagger}$, then $\#\gamma^{-1}(Y)=\infty$. Hence, after taking a subseqence one can find a strictly monotone sequence $t_i\in [a_\gamma,b_\gamma]$ such that $\gamma(t_i)\in Y$, $t_i\rightarrow t_0\in [a_\gamma,b_\gamma]$ and $\gamma(t_i)\rightarrow \gamma(t_0)\in Y$. In a blow up of $Y$ around $\gamma(t_0)$ the sequence $\gamma(t_i)$ converges to the the velocity vector of $\gamma$ at $t_0$. Hence, we conclude that $\dot{\gamma}(t_0)\in T_{\gamma(t_0)}Y$ and $\gamma$ is tangent to $Y$. 
\end{proof}

For $U\subset X$ open we write
\begin{align*}
{\mathcal H^n({\T_u\cap U})}=\int_Q\m_{\gamma}(U) d\mathfrak q(\gamma) = \int_{g^{-1}(U)} h_\gamma(r) dr\otimes d\mathfrak q(\gamma)
\end{align*}
where $g:\mathcal V\subset \R\times Q\rightarrow \T^b_u$ is the ray map defined in Subsection \ref{subsec:1Dlocalisation}. We also note that $(r,\gamma)\in \mathcal V \mapsto h_\gamma(r)$ is measurable. 

\begin{remark}
Let $B\subset \R\times Q$ be measurable. Then $g(B\times Q) =:\mathcal B\subset X$ is a measurable subset since $g$ is a Borel isomorphism. Then $\mathcal B\cap X_\gamma$ is measurable w.r.t. the induced measurable structure and by Fubini's theorem 
the map 
$$\gamma\in Q\mapsto L(\gamma|_{\gamma^{-1}(\mathcal B)})$$
is measurable.
We can {apply this} for the case when $B=(-\infty,0)\times Q$. It follows that
$\gamma\in Q\mapsto a_\gamma =L(\gamma|_{\gamma^{-1}(g((-\infty,0)\times Q))} \in \R$ is measurable. Similar for $b_\gamma$.
\end{remark}
\begin{remark}
Consider the map $\Phi_t: \R\times Q\rightarrow \R\times Q$, $\Phi_t(r,q)=(tr,q)$ for $t>0$. Then, it is clear that $\Phi_t(\mathcal V)=\mathcal V_t$ is a measurable subset of $\mathcal V$ for $ t\in (0,1]$. Moreover $g(\mathcal V_t)=\T^b_{u,t}$ is a measurable subset of $\T^b_u$ such that $X_\gamma\cap \T^b_{u,t}= t X_\gamma \subset (a_\gamma,b_\gamma)$. If $t\in (0,1)$, then $\mathcal H^n(\T^b_u\backslash \T^b_{u,t}) >0$.

Again by Fubinis theorem  $U\cap X_\gamma \cap \T_{u,t}^b= U\cap t X_\gamma$ is measurable in $X_\gamma$ for $\mathfrak q$-a.e. $\gamma\in Q$ and the map $$L_{U,t}:\gamma\in Q\mapsto  \mbox{L}(\gamma|_{(t a_\gamma, tb_\gamma)\cap \gamma^{-1}(U)})=\int 1_{U\cap tX_\gamma} d\mathcal L^1$$ is measurable. We note that the set $(ta_\gamma, tb_\gamma)\cap \gamma^{-1}(U)$ might not be an interval.

\end{remark}

Let $Y\subset X$ and consider $U_\eps=B_{\epsilon}(Y)$ for $\epsilon>0$. For $s\in \N$ and $t\in (0,1]$ we define
\begin{align*}
C_{\epsilon, s, t}=\left\{\gamma\in Q:  \mbox{L}(\gamma|_{\gamma^{-1}(U_\eps)\cap (ta_{\gamma},tb_\gamma)})>\epsilon s\right\}.
\end{align*} 
Further, we set 
\begin{align*}
C_{s,t}=\bigcup_{\epsilon>0}\bigcap_{\epsilon'{\leq}\epsilon} C_{\epsilon',s,t}=\{\gamma\in Q: \liminf_{\epsilon\rightarrow 0}\mbox{L}(\gamma|_{\gamma^{-1}(U_\epsilon)\cap (ta_\gamma, tb_\gamma)})/\epsilon \geq s\}
\end{align*}
and 
\begin{align*}
 C_t=\bigcap_{s\in \N}C_{s,t} =\{\gamma\in Q: \lim_{\epsilon\rightarrow 0}\mbox{L}(\gamma|_{\gamma^{-1}(U_\epsilon)\cap (ta_\gamma, tb_\gamma)})/\epsilon=\infty\}.
\end{align*}

\begin{lemma}\label{lemma: tan-t->Ct}
Let  $0<t\le 1$ and let $\gamma\in Q$. If $\gamma|_{[ta_\gamma, tb_\gamma]}$ is tangent to $Y$, then $\gamma\in C_t$.
\end{lemma}
\begin{proof}
For the proof we ignore $t\in (0,1]$ and consider $\gamma|_{[a_\gamma,b_\gamma]}$.

Let $\gamma\in Q$ be tangent to $Y$. Assume $\gamma\notin C$.
Then there exists a sequence $(\epsilon_i)_{i\in \N}$ such that $\lim_{i\rightarrow \infty} {L(\gamma|_{\gamma^{-1}(B_{\epsilon_i}(Y))\cap (a_\gamma,b_\gamma)})}/{\epsilon_i}=C\in [0,\infty)$. 
By assumption there exists $t_0\in [a_\gamma, b_\gamma]$ such that $\gamma(t_0)\in Y$. Hence $t_0\in \gamma^{-1}(B_{\epsilon_i}(Y))\cap [a_\gamma,b_\gamma]$. 
There exists a maximal interval $I^{\epsilon_i}$ contained in $\gamma^{-1}(B_{\epsilon_i}(Y))\cap [a_\gamma,b_\gamma]$  such that $t_0\in I^{\epsilon_i}$. 
After taking another subsequence we still have
$$\infty>\lim_{i\rightarrow \infty} \frac{L(\gamma|_{\gamma^{-1}(B_{\epsilon_i}(Y))\cap (a_\gamma,b_\gamma)})}{\epsilon_i}\geq \lim_{i\rightarrow \infty }\frac{L(\gamma|_{I^{\epsilon_i}\cap (a_\gamma,b_\gamma)})}{\epsilon_i}=:C'\geq 0$$ 
and  $L(\gamma|_{I^{\epsilon_i}\cap (a_\gamma,b_\gamma)})=:L_i\rightarrow 0$. We set $\gamma_i=\gamma|_{I^{\epsilon_i}\cap (a_\gamma,b_\gamma)}$. Since $I^{\epsilon_i}$ is maximal such that $\mbox{Im}(\gamma_i)\subset B_{\epsilon_i}(Y)$, we have 
$\sup_{y\in Y, t\in I^{\epsilon_i}}d(y,\gamma(t))\geq \epsilon_i.$
In the rescaled space $(Z, \frac{1}{L_i} d_Z)$ the geodesic $\gamma_i$ is a geodesic of length $1$ and $$\sup_{y\in Y, t\in I_{\epsilon_i}}\frac{1}{L_i}d(y,\gamma_i(t_0))\geq C'/2$$
for $i\in \N$ suffienciently large.
By the proof of Petrunin's glued space theorem  we know that $(Y, \frac{1}{L_i}d|_{Y})$ converges in GH sense to $T_{\gamma(t_0)}Y$.
Therefore, for $\epsilon_i\rightarrow 0$ a sequence of points $\gamma_i(t_i)$ converges to $v\in T_{\gamma(t_0)}Z$ such that $\sup_{w\in T_{\gamma(t_0)}Y} \angle (w,v)\geq C'/2$. This is a contradiction since the tangent vector $\dot{\gamma}(t_0)\in T_{\gamma(t_0)} Y$ and tangent vector of geodesics in Alexandrov spaces are well-defined and unique.

Hence, for any sequence $\epsilon_i\rightarrow 0$ it follow that $\frac{L_i}{\epsilon_i}\rightarrow \infty$ and therefore $\gamma \in C_t$. %
\end{proof}
\begin{corollary}\label{cor: tan->Ct}
Let $\gamma\in Q$. If {$\gamma|_{(a_\gamma,b_\gamma)}$} is tangent to $Y$, then $\gamma\in {C= \bigcup_{t\in (0,1)}}C_t$. 
\end{corollary}
\begin{lemma}\label{lemma:tangentgeodesics}
Let $\mathfrak q$ be associated to $u$. Then $\mathfrak q(Q\cap \bigcup_{t\in (0,1)} C_t)=0$.
\end{lemma}

\begin{proof} It is clearly enough to show that for any  $t\in (0,1)$ it holds that $\mathfrak q(Q\cap C_t)=0$. Therefore in the following we work with a fixed $t$.

We recall that {$a_\gamma<0<b_\gamma$},  $\gamma\in Q\mapsto a_\gamma, b_\gamma$ are measurable and $Q=\bigcup_{l\in \N}\{l\geq |b_\gamma|, |a_\gamma| \geq \frac{1}{l}\}$.
It is obviously enough to prove the lemma for $Q^l=\{l\geq |a_\gamma|, |b_\gamma|\geq \frac{1}{l}\}$  for arbitrary $l\in \N$. Therefore
we fix $l\in N$ and replace $Q$ with $Q^l$. By abuse of notation we will drop the superscript $l$ for the rest of the proof.
By rescaling the whole space with $4l$ we can assume that {$4\le |a_\gamma|, |b_\gamma| \le 4l^2$ }for each $\gamma\in Q$.

Let $C_{\epsilon, s, t}$ be defined as  before for $\epsilon\in (0,\epsilon_0)$ and $s\in \N$.

We pick $\gamma\in C_{\epsilon,s, t}$ and consider $\gamma^{-1}(B_{\epsilon}(Y))\cap (ta_\gamma, tb_\gamma)=: I_{\gamma, \epsilon}$. We set $L(\gamma|_{I_{\gamma, \epsilon}})=:L^{\epsilon}$.

We observe that $$4l^2\geq (1-t)|a_\gamma|\geq (1-t)4,\ \ \  4l^2\geq (1-t)|b_\gamma| \geq (1-t)4.$$ 
We pick $r\in I_{\gamma, \epsilon}$ and $\tau \in (a_\gamma, ta_\gamma)\cup (tb_\gamma, b_\gamma)$. Theorem \ref{th:1dlocalisation} implies that $([a_\gamma, b_\gamma], h_\gamma dr)$ satisfies the condition $CD(k(n-1),n)$. Then, the following estimate holds (c.f. \cite[Inequality (4.1)]{cavmon})
\begin{align*}
h_\gamma(r)&\geq \frac{\sin^{n-1}_{k}((r-a_\gamma)\wedge (b_\gamma-r))}{\sin^{k-1}_{k}((\tau-a_\gamma)\wedge (b_\gamma-\tau))}{h_\gamma(\tau)} \\ &\geq \frac{\sin^{n-1}_{k}((1-t)4)}{\sin^{n-1}_{k}4l^2} h_\gamma(\tau)= C(k,n,t,l) h_\gamma(\tau).
\end{align*}
for a universal constant $C(k,n,t,l)$.
We take the mean value w.r.t. $\mathcal L^1$ on both sides and obtain
\begin{align*}
\frac{1}{L^{\epsilon}}\int_{I_{\gamma, \epsilon}} h_\gamma d\mathcal L^1 \geq C(k,n,t,l) \frac{1}{4l^2}\int_{(a_\gamma, ta_\gamma)\cup (tb_\gamma, b_\gamma)} h_\gamma d\mathcal L^1.
\end{align*}
Hence, after integrating w.r.t. $\mathfrak q$ on $C_{\epsilon,s,t}$ and taking into account $\frac{1}{\epsilon s}\geq \frac{1}{L^{\epsilon}}$ by definition of $C_{\epsilon,s,t}$, it follows
\begin{align*}
\frac{1}{\epsilon s}\mathcal H^n(B_{\epsilon}(Y))&\geq 
\frac{1}{s\epsilon}\int_{C_{\epsilon,s,t}}\m_{\gamma}(B_{\epsilon}(Y))d\mathfrak q(\gamma)\\
&\geq\frac{1}{L^{\epsilon}}\int_{C_{\epsilon,s,t}} \int_{I_{\gamma, \epsilon}} h_\gamma d\mathcal L^1 d\mathfrak q(\gamma) \\
&\geq \hat C \int_{C_{\epsilon,s,t}}\int_{(a_\gamma, ta_\gamma)\cup (tb_\gamma, b_\gamma)} h_\gamma d\mathcal L^1 d\mathfrak q(\gamma)\\
&\geq \hat C \int_{C_{\epsilon,s,t}} \m_\gamma(\T^b_u\backslash \T^b_{u,t})d\mathfrak q(\gamma)
\end{align*}
where $\hat C=\frac{1}{2l} C(k,n,t,l)$.
It is known that $\mathcal H^n(B_{\epsilon}(Y))\leq \epsilon M$ for some constant $M>0$ provided $\epsilon>0$ is sufficiently small. {
This follows from semiconcavity of the boundary distance function in Alexandrov spaces, Lipschitz continuity of the induced gradient flow and the coarea formula. }
Hence
\begin{align*}
\frac{M}{s} \geq C(K,N,k,t) \int_{C_{\epsilon,s,t}} \m_\gamma(\T^b_u\backslash \T^b_{u,t})d\mathfrak q(\gamma).
\end{align*}
If we take limit for $\epsilon\rightarrow 0$, we obtain

\begin{align*}
\frac{M}{s} \geq C(K,N,k,t) \int_{C_{s,t}} \m_\gamma(\T^b_u\backslash \T^b_{u,t})d\mathfrak q(\gamma).
\end{align*}
Finally, for $s\rightarrow \infty$ it follows
\begin{align*}
0 = \int_{C_{t}} \m_\gamma(\T^b_u\backslash \T^b_{u,t})d\mathfrak q(\gamma).
\end{align*}
But by construction of $\T^b_{u,t}$ we know that $\m_\gamma(\T^b_u\backslash \T^b_{u,t})$ is positive for every $\gamma\in Q$ if $t\in (0,1)$. Therefore, it follows $\mathfrak q(C_t)=0$.
\end{proof}
Combining the above lemma with Corollary~\ref{cor: tan->Ct} gives
\begin{corollary}\label{cor: tan-measure 0}
Let $\mathfrak q$ be associated to $u$. Then $\mathfrak q(\gamma\in Q:  \gamma|_{(a_\gamma,b_\gamma)}$ is tangent to $Y)=0$.
\end{corollary}
Let us remark here that we \emph{ do not} claim that the set of geodesics in $Q$  which are tangent to $Y$ at one of the endpoints has measure zero. We suspect this is true but this is not needed for the applications.

As a first consequence of Proposition \ref{prop} and Corollary \ref{cor: tan-measure 0} we obtain the following corollary.

\begin{corollary}\label{cor:important}
Let $x_1\in X_1\backslash Y$. Then, {for} $\mathcal H^n_Z$-a.e. point $x_0\in X_0$ the geodesic that connects $x_0$ and $x_1$ intersects with $Y$ only finitely many times and in any intersection point $\Phi_Z|_Y$ is differentiable.
\end{corollary}

\section{Semiconcave functions on glued spaces}

\begin{lemma} \label{gluingcondition}
 Let $(X,d)$ be an $n$-dimensional Alexandrov space and let $\Phi: X\rightarrow \R$ be a double semi-concave function. 

Then $\Phi$ is semi-concave in the usual sense and for any $p\in \partial X$ it holds that $d\Phi(v)\leq 0\ \mbox{ for any normal vector }v\in \Sigma_p. $
\end{lemma}
\begin{proof}
A function $\Phi: X\rightarrow \R$ is semi-concave if $\Phi\circ P$ is semi-concave on the double space $\hat X$ that is an Alexandrov space without boundary. In particular, for any geodesic $\gamma$ in $\hat X$ such that $\mbox{Im}(\gamma)\subset X$ the composition $\Phi\circ \gamma$ is semi-concave.

Moreover, let $p\in X$ be a boundary point such that there exists $v\in \Sigma_pX$  normal to the boundary. Considering $p$ in $\hat X$ we know that $\Sigma_p\hat X= \widehat{\Sigma_pX}$. Hence $v,-v\in \Sigma_p\hat X$ and $\angle (v,-v)=\pi$. Therefore, $v$ and $-v$ generate a geodesic line in $T_p\hat X= C(\Sigma_p\hat X)$. Since $\Phi$ is double semi-concave, its differential  $d\Phi_p:T_p\hat X\rightarrow \R$ is concave, and by Lemma \ref{importantlemma} it follows that 
$d\Phi(-v)\geq d\Phi(v)$. On the other hand we have $d\Phi(-v)=-d\Phi(v)$. This implies the claim.
\end{proof}

We continue to work with the setup from the previous section.  Let $\Phi_i:X_i\rightarrow \R$, $i=0,1$, be semi-concave such that they agree on the boundaries $\partial X_0$ and $\partial X_1$ respectively identified via an isometry $\mathcal I$.   Let $Z$ be the glued space and let $\Phi_Z\co Z\to\R$ be the naturally constructed glueing of $\Phi_0,\Phi_1$. 
\begin{lemma}\label{lem:important}
Let $\gamma:[0,L(\gamma)]\rightarrow Z$ be a constant speed geodesic with $\gamma(0)\in X_0\backslash \partial X_0$ and $\gamma(L(\gamma))\in X_1\backslash \partial X_1$.
Suppose $\gamma$ intersects $Y$ in a single point $p=\gamma(t_0)$. Suppose further that the following conditions hold:

\begin{enumerate}[(i)]
\item $p$ is a regular point in $Z$;
\medskip
\item $
d\Phi_0|_p(v_0)+ d\Phi_1|_p(v_1)\leq 0\ \ \forall\mbox{ normal vectors }v_i\in \Sigma_pX_i,\  i=0,1;
$
\medskip
\item The restriction $\Phi|_Y$ is differentiable at $p$.
\end{enumerate}
Then $\Phi_Z\circ \gamma:[0,L(\gamma)]\rightarrow \R$ is semi-concave. In particular
\begin{align}\label{inequ:important}
-d\Phi_0(\dot \gamma^-)=\frac{d^-}{dt}\Phi_0\circ \gamma(t_0)\geq \frac{d^+}{dt}\Phi_1\circ \gamma(t_0)= d \Phi_1(\dot \gamma^+).
\end{align}
where $t\in [0,L(\gamma)]\mapsto \gamma^-(t)=\gamma(L(\gamma)-t)$ and $\gamma^+=\gamma$.
\end{lemma}

\begin{proof} By Lemma \ref{importantlemma} we only need to check  \eqref{inequ:important}. In the following we write $\gamma^{+/-}$ instead of $\dot \gamma^{+/-}$. By assumption we have that  $\gamma([0,t_0])\subset X_0$ and $\gamma([t_0,L(\gamma)])\subset X_1$. 
By assumption 
$\gamma(t_0)=:p$ is a regular point in $Z$, that is $T_pZ=\R^n$ and $\Sigma_pZ=\mathbb{S}^{n-1}$. 
Moreover $\Sigma_pZ$ is the glued of $\Sigma_p{X_0}$ and $\Sigma_p{X_1}$ along their isometric boundary and $\partial \Sigma_pX_0=\mathbb S^{n-1}_+$ and $\partial \Sigma_pX_1=\mathbb S^{n-1}_-$ (see the remarks at the end of Subsection \ref{subsec:gluing}). In this case the north pole $N$ and the south pole $S$ are the unique normal vectors in $\Sigma_pX_0$ and $\Sigma_pX_1$, respectively.

If $\gamma^-(t_0)=N\in \Sigma_pX_0$, then by symmetry $\gamma^+(t_0)=S$ and by assumption it follows
\begin{align*}
-d\Phi_0(\gamma^-)\geq 0\geq d\Phi_1(\gamma^+).
\end{align*}

If  $\gamma^-(t_0)\neq N$, then 
it 
also follows $\gamma^+(t_0)\neq S$.
There exists a geodesic loop in $\Sigma_pZ$ that contains $\gamma^-(t_0)$ and $\gamma^+(t_0)$, and intersects with 
$\partial \Sigma_p{X_0}=\Sigma_p Y=\mathbb S^{n-2}$ twice in  $w^-$ and $w^+$ such that $\angle(w^-,w^+)=\pi$.

Since $\angle(w^+,w^-)=\pi$, there exists a geodesic line in $T_pX$ of the form $s\in \R \mapsto (-s w^- )\star( s w^+)$  passing through $0$ where $\star$ denotes the concatenation of curves. 
Since we assume $\Phi_Z|_Y$ is differentiable in $p$,  $d\Phi_Z:T_pX_i\rightarrow \R$ is linear and 
% 
% since the right and left derivative along this line at $0$ are $d\Phi_0(\gamma^+)$ and $d\Phi_1(\gamma^-)$, it follows that $ -d\Phi_0(w^-)\geq d\Phi_0(w^+)$, and
therefore $d\Phi_Z(w^-)+ d\Phi_Z(w^+)=0$.

Let $\sigma^0$, $(\sigma^1)^{-1}: [0,\pi]\rightarrow \Sigma_p{X_0},\Sigma_p{X_1}$ be the geodesics in $\Sigma_p{X_0}$ and $\Sigma_p{X_1}$ respectively connecting $w^-$ and $w^+$ such that their concatenation is the geodesic loop $S$ in $\Sigma_p$ through $w^+, w^-, \gamma^+$ and $\gamma^-$.

Since $d\Phi_i: T_pX_i\rightarrow \R$ is concave, $T_pX_i$ is the metric cone over $\Sigma_p{X_i}$ and $d\Phi_i (rv)= r d\Phi_i(v)$, it follows that $d\Phi_i\circ \sigma_i=u^i: [0,\pi]\rightarrow \R$ satisfies 
\begin{align*}
(u^i)''+ u^i\leq 0, \ i=0,1.
\end{align*}
Moreover $v(s)= d\Phi_0 \circ \sigma_0(s)+ d\Phi_1 \circ \sigma_1(s)$ satisfies 
\begin{align*}
v'' + v\leq 0, \ \ v(0)= v(\pi)=d\Phi_0(w^-)+ d\Phi_1(w^+)= 0.
\end{align*}

Let $r,s$ be the polar coordinates on $\R^2_+=\{(x,y): y\ge 0\}$ with $x=r\cos s, y=r\sin s$. 
Then the function $f(x,y)=rv(s)$ is concave. Since $f(-1,0)=f(1,0)=0$ and $f(0,1)\le 0$ concavity of $f$ implies that $f\le 0$.
Therefore $v\le 0$ and hence  $d\Phi_0(\gamma^-)+ d\Phi_1(\gamma^+)\le 0$.

\end{proof}

\begin{theorem}
 
%Let $\gamma:[0,L(\gamma)]\rightarrow Z$ be a constant speed geodesic.
 Let $\Phi_i:X_i\rightarrow \R$, $i=0,1$, be semi-concave  such that for any $p\in \partial X_i$ it holds that
\begin{align*}
d\Phi_0|_p(v_0)+ d\Phi_1|_p(v_1)\leq 0\ \ \forall\mbox{ normal vectors }v_i\in \Sigma_pX_i,\  i=0,1.
\end{align*}
%Let $N\subset Y$ be a set such $\Phi_i|_Y$ is differentiable on $Y\backslash N$ and $\mathcal{H}^{n-1}(N)=0$.
Then $\Phi_Z: Z\rightarrow \R$ is semiconcave. 
\end{theorem}

\begin{proof} 

It is obvious that we only need to check semi-concavity of $\Phi_Z$ near $Y$. By changing $Z$ to a small convex neighborhood of a point $p\in Y$ we can assume that $\Phi_i$ are $\lambda$-concave on $X_i$ for some real $\lambda$.

Let $\gamma\co [0,L]\to Z$ be a unit speed geodesic. We wish to prove that $\Phi_Z(\gamma(t))$ is $\lambda$-concave. Fix an arbitrary $0<\delta<L/10$ and let $ x_0=x_1=\gamma(\delta), x_1=\gamma(L-\delta)$.  Let $y_i\to x_0, z_i\to x_1$ be such that $y_i,z_i\notin Y$ for any $i$. Let $\gamma_i$ be a shortest unit speed geodesic from $y_i$ to $z_i$.
By  Corollary \ref{cor:important}  we can adjust $z_i$ slightly so that  that each $\gamma_i$ intersects $Y$ at most finitely many times, all intersection points are regular and $\Phi_Z|_Y$ is differentiable at those intersection points. Therefore by Lemma ~\ref{lem:important} we have that $\Phi_Z|_{\gamma_i}$ is $\lambda$-concave for every $i$. By passing to a subsequence we can assume that $\gamma_i$ converge to a shortest geodesic from $x_0$ to $x_1$ and since Alexandrov spaces are nonbranching this geodesic must be equal to $\gamma|_{[\delta,L-\delta]}$. Therefore by continuity of $\Phi_Z$ we get that $\Phi_Z$ is $\lambda$-concave on $\gamma|_{[\delta,L-\delta]}$. Since this holds for arbitrary $0<\delta<L/10$  we conclude that $\Phi_Z$ is $\lambda$-concave on all on $\gamma$.
\end{proof}

\begin{corollary}\label{cor:gluingfct}
A function $\Phi: X\rightarrow \R$ is double semi-concave if and only if it is semi-concave in the usual sense and for any $p\in \partial X$ it holds that $d\Phi(v)\leq 0\ \mbox{ for any normal vector }v\in \Sigma_p. $
\end{corollary}

 Let $\m_i=\Phi_i\mathcal{H}^n_{X_i}$ be measures on $X_0$ and $X_1$, respectively, for semi-concave function $\Phi_0$ and $\Phi_1$, and assume $\Phi_0|_{\partial X_0\equiv \partial X_1}=\Phi_1|_{\partial X_0\equiv \partial X_1}$.

Then, the metric measure glued space between the weighted Alexandrov spaces $(X_i, d_{X_i},\m_i)$, $i=0,1$, is given by 
\begin{align*}
(X_0\cup_{\mathcal I} X_1, \m_Z) \mbox{ where } \m_Z=(\iota_0)_\# \m_0 + (\iota_1)_\# \m_1.
\end{align*}
The maps $\iota_i:X_i\rightarrow Z$, $i=0,1$, are the canonical inclusion maps.
Note that $X_0\cup_{\mathcal I} X_1$ is an $n$-dimensional Alexandrov space by Petrunin's glued space theorem. By Remark \ref{rem:local} it follows that 
\begin{align*}
\left(\mathcal{H}^n_{X_0\cup_{\mathcal I}X_1}\right)|_{X_i}=\mathcal{H}^n_{X_i}, \ i=0,1
\end{align*}
we can write $\m_Z=\Phi_Z\mathcal{H}^n_{X_0\cup_{\mathcal I}X_1}$.

\section{Proof of Theorem  \ref{main}}
In this section we present the proof of the glued space theorem (Theorem \ref{main}).

\noindent
{\it Proof of Theorem \ref{main}}

{\textbf 1. } Let $X_i$, $i=0,1$, be Alexandrov spaces with curvature bounded from below by  $k_0$ and $k_1$, respectively.

By Theorem \ref{th:petruningluing} it follows that $X_0\cup_{\phi}X_1=:Z$ has curvature bounded from below by $\min\{k_0,k_1\}=:k$.

By Theorem \ref{th:petrunincd} the metric measure space $(Z,d_Z,\mathcal{H}^n_Z)$ satisfies the condition $CD(k(n-1),n)$.

Hence, any $1$-Lipschitz function $u: (Z,d_Z)\rightarrow \R$ induces a disintegration $\{\m_\gamma\}_{\gamma\in Q}$ that is strongly consistent with $R^b_u$, and for $\mathfrak q$-a.e. $\gamma\in Q$ the metric measure space $(\overline{X_\gamma},\m_\gamma)$ satisfies the condition $CD(k(n-1),n)$ and hence $CD(k(n-1),N)$ by monotinicity in $N$ . It follows that $\m_{\gamma}=h_\gamma \mathcal H^1|_{X_\gamma}$ and $h_\gamma:[a_\gamma, b_\gamma]\rightarrow \R$ satisfies 
\begin{align*}
\frac{d^2}{dr^2} h_\gamma^{\frac{1}{N-1}}+ k h_\gamma^{\frac{1}{N-1}}\leq 0 \ \mbox{ on } (a_\gamma, b_\gamma) \ \mbox{ for }\mathfrak q\mbox{-a.e.}\gamma\in Q.
\end{align*}
By Lemma \ref{importantlemma} it follows
\begin{align*}
\frac{d^-}{dr} h_{\gamma}^{\frac{1}{N-1}}\geq \frac{d^+}{dr} h_\gamma^{\frac{1}{N-1}} \ \text{ everywhere on }(a_\gamma, b_\gamma).
\end{align*}
{\bf 2.}  Fix $0<t<1$. Define the set $C_t$ as in Section~\ref{sec: Second application}. 
%We have observed that $\#\gamma|_{(ta_\gamma, tb_\gamma)}^{-1}(Y)<\infty$ for every $\gamma\in Q\backslash C_t$ and $\mathfrak q(Q\cap C_t)=0$. 
%So we replace $Q$ with $Q\backslash C_t$.

Recall that regular points have full measure in $Z$. Hence, there exists $\hat Q\subset Q$ with full $\mathfrak q$-measure such that $\gamma(r)$ is a regular point for any $r\in [0,L(\gamma)]$ and for every $\gamma\in \hat Q$.
Let $Q_t=\hat Q\backslash C_t$. By Lemma~\ref{lem: infinite implies tangent} and Lemma~\ref{lemma: tan-t->Ct} we know that any $\gamma\in Q_t$ it holds that $\gamma|_{(ta_\gamma,tb_\gamma)}$ intersects $Y$ in finitely many points. Further by Lemma ~\ref{lemma:tangentgeodesics} we know that $Q_t$ has full measure in $Q$.
%We replace $Q$ with $\hat Q\backslash C$, and  by abuse of notation we write $\hat Q\backslash C=:Q$ in the following.

Let $p\in X_0\backslash \partial X_0$ be arbitrary. By construction of the glued metric we can pick $\epsilon>0$ that is sufficiently small such that $d_Z|_{B_\epsilon(p)\times B_\epsilon(p)}=d_{X_0}|_{B_\epsilon(p)\times B_\epsilon(p)}$. Moreover, since $(X_0,d_{X_0})$ is an Alexandrov space there exists an open domain $U_p\subset B_{\epsilon}$ that is geodesically convex. 
There is a countable set of points $\{p_i:i\in\N\}$ such that $\bigcup_{i\in \N} U_{p_i}={ X_0\backslash Y}$. We pick $i\in \N$ and consider the corresponding $U_{p_i}$. In the following we drop the subscript $p_i$ and work with $U=U_{p_i}$.
Convexity of $U$ implies that $(\overline{U}, d_{X_0}|_{\overline{U}\times \overline{U}}, \m_{X_0}|_{\overline{U}})$ satisfies the condition $CD(K,N)$, $u|_{\overline{U}}$ is $1$-Lischitz and the set $\T_u\cap \overline{U}=\tilde \T_u$ is the transport set of $u$ restricted to $\overline{U}$.

We obtain a decomposition of $\overline{U}$ via $X_\gamma\cap \overline{U}=\tilde X_\gamma$. The subset $\mathfrak Q(U)=\tilde Q\subset Q$ of geodesics in $Q$ that intersect with ${U}$ is measurable. We can  pushforward the measure $\m|_{\overline{U}}$
 w.r.t. the quotient map $\mathfrak Q: \overline{U}\rightarrow \tilde Q$ and we obtain a measure $\tilde{\mathfrak q}$ on $\tilde Q$. 
By the $1D$-localisation procedure applied to the metric measure space $\overline{U}$, there exists a disintegration $(\tilde \m_{\tilde\gamma})_{\gamma\in \tilde Q}$ where the geodesic $\tilde \gamma$ is defined as intersection of $X_\gamma$ with $\overline{U}$.
% via  $\mbox{Im}(\tilde \gamma)=\mbox{Im}(\gamma)\cap \overline{U}=\tilde X_\gamma$. 
We also set $\mbox{Im}(\tilde \gamma)=:X_{\tilde \gamma}$. Moreover, for $\tilde{\mathfrak q}$-a.e. $\tilde \gamma$ the metric measure space $(X_{\tilde\gamma}, \tilde\m_{\tilde\gamma})$ is $CD(K,N)$.
That is,  there exists a density $\tilde h_{\tilde \gamma}$ of $\tilde\m_{\tilde\gamma}$ w.r.t. $\mathcal H^1$ such that 
\begin{align}\label{inequ:another}
\frac{d^2}{dr^2} \tilde h_{\tilde \gamma}^{\frac{1}{N-1}}+ \frac{K}{N-1} \tilde h_{\tilde \gamma}^{\frac{1}{N-1}}\leq 0 \mbox{ on } (a_{\tilde\gamma},b_{\tilde\gamma})\subset (a_\gamma,b_\gamma) \ \mbox{ for }\tilde{\mathfrak q}\mbox{-a.e.}\ .
\end{align}
More precisely, there exists a set $\mathcal N\subset \tilde Q$ with $\tilde{\mathfrak q}(N)=0$ such that \eqref{inequ:another} holds for every $\tilde\gamma\in \tilde Q\backslash N$.

{\bf 3.} We show that $\tilde{\mathfrak q}$ is absolutely continuous w.r.t. $\mathfrak q$ on $Q$. 

Recall $\tilde{\mathfrak q}= (\mathfrak Q)_{\#}\m|_{U}$. Let $A\subset Q$ be a set such that $\mathfrak q(A)=0$. 
Hence $0=\m(\mathfrak Q^{-1}(A))\geq \m(\mathfrak Q^{-1}(A)\cap U)$. Hence $\tilde{\mathfrak q}(A)=0$.

Therefore, there exists a measurable function $G:Q\rightarrow [0,\infty)$ such that $ \mathfrak q= G \tilde{\mathfrak q}$ and $\int_{Q\backslash \mathfrak Q^{-1}(U)} G d\tilde{\mathfrak q}=0$. In particular, it follows that $\mathfrak q(
\mathcal N)=0$.

A unique and strongly consistent disintegration of $\m_Z|_{\mathcal T^b_u}=\Phi \mathcal H^n_Z|_{\mathcal T^b_u}$ is given by $$\int_Q \Phi \m_\gamma d\mathfrak q$$ where $\Phi \m_\gamma= (\gamma)_{\#}[\Phi\circ \gamma h_\gamma\mathcal H^1]$.
Then, it follows by uniqueness of the disintegration and since $\mathfrak q=G\mathfrak q$ that $G(\gamma)\tilde h_{\tilde\gamma}=(\Phi\circ\gamma) h_\gamma$ on $(a_{\tilde \gamma}, b_{\tilde\gamma})$ for $\tilde{\mathfrak{q}}$-a.e. $\tilde\gamma$. 

{\bf 4.} We repeat the steps {\bf 2.} and {\bf 3.} for any $U_{p_i}$, $i\in \N$. We can find a set $\mathcal N\subset Q$ with $\mathfrak q(\mathcal N)=0$ such that $\tilde{\mathfrak q}_{p_i}(\mathcal N)=0$ for every $i\in \N$ and such that \eqref{inequ:another} holds for any $\gamma \in \mathfrak Q^{-1}(U_{p_i})\backslash \mathcal N$ for any $i\in \N$.

We repeat all the previous steps again for $X_1$ instead of $X_0$ and find a correponding set $\mathcal N\subset Q$ of $\mathfrak q$-measure $0$.

We get that  for every $\gamma\in Q\backslash \mathcal N$ the inequality \eqref{inequ:another} holds for $h_\gamma$ for any interval $I\subset (a_\gamma, b_\gamma)$ as long $\gamma|_I$ is fully contained in $U_{p_i}$ for some $i\in \N$.

From Lemma \ref{importantlemma} and Lemma \ref{lem:important} it follows that inequality \eqref{inequ:another} holds for $\Phi\circ \gamma h_\gamma$ on $(ta_\gamma, tb_\gamma)$ for any $\gamma\in Q_t\backslash N$.
Since this holds for arbitrary $0<t<1$, we get that for $q$-almost all $\gamma$ in $Q$ it holds that $([a_\gamma,b_\gamma], m_\gamma)$ satisfies $CD(K,N)$.
Since this holds for an arbitrary 1-Lipschitz function $u$ we obtain that $Z$ satisfies $CD^1_{lip}(K,N)$.

If $\m_Z$ is a finite measure Theorem  \ref{thm:cavmil} yields the condition $CD(K,N)$ for $(Z,d_Z,\m_Z)$. 

If $\m_Z$ is a $\sigma$-finite measure we argue as follows. 
For $\overline{U}$ that is a geodesically convex and closed neighborhood with finite measure of some point $x\in Z$, it holds that the metric measure space $(\overline{U}, d_Z|_{\overline{U}\times \overline{U}}, \m_Z|_{\overline{U}})$ satisfies $CD^1_{lip}(K,N)$. Hence, by 
Theorem \ref{thm:cavmil} it satisfies $CD(K,N)$ and also $CD^*(K,N)$. Finally by the globalisation theorem of $CD^*$ \cite{bast} the space $(Z,d_Z,\m_Z)$ satisfies $CD^*(K,N)$. 
\qed

\begin{example}\label{example}
Here we give another simple example that shows why Theorem \ref{main} fails for the measure contraction property $MCP$. 

We consider a metric space $Z$ that is the cylinder $[0,\frac{3}{4}\epsilon]\times \mathbb S^{N-1}_{\delta}$ for $0<\delta\ll \frac{\epsilon}{8}$ with one end closed by a disk. This space has nonnegative Alexandrov curvature and equipped with the $N$-dimensional Hausdorff measure  is $CD(0,N)$ by Petrunin's theorem. It has diameter less than $\epsilon>0$. 

In  \cite{stugeo2} (Remark 5.6) it was observed that there exists a constant $c_{N+1}\in (0,1]$ such that $\forall \theta>0$ with $N\theta^2\leq c_{N+1}$ it holds
\begin{align}
t^{N}\geq \tau_{N,N+1}^{(t)}(\theta)^{N+1} \ \forall t\in (0,1).
\end{align}
Hence, provided $N\epsilon^2\leq c_{N+1}$, $Y$ will satisfy the $MCP(N,N+1)$.  

We show that if we pick $\epsilon$ sufficiently large, the double space does not satify this property. The function $\theta \mapsto \tau_{N,N+1}^{(t)}(\theta)^{N+1}$ is monotone increasing and $\tau_{N,N+1}^{(t)}(\theta)^{N+1}\rightarrow \infty$ for all $t\in (0,1)$ if $\theta\uparrow \pi$. Therefore, the set 
\begin{align*}
\Theta=\{\theta >0: t^N\geq \tau_{N,N+1}^{(t)}(\theta)^{N+1}\ \forall t\in (0,1)\}
\end{align*}
is nonempty and bounded by $\pi$ and for $\theta \in \Theta$ and $\theta'\leq \theta$ it holds $\theta'\in \Theta$.
We pick $\Theta \ni \epsilon\geq\frac{8}{9}\sup \Theta$ in the construction above. Then, by definition of $\Theta$ the space $Y$ will satisfy $MCP(N,N+1)$.
The double space of $Y$ is the cylinder $[0, \frac{3}{2}\epsilon]\times \mathbb{S}^{N-1}_\delta$ with both ends closed by a disk. Since $\frac{3}{2}\epsilon\geq \frac{4}{3}\sup \Theta$, it follows $t^N< \tau_{N,N+1}^{(t)}(\frac{5}{4}\epsilon)^{N+1}$ for some $t\in (0,1)$.

On the other hand,  since $[0,\frac{3}{2}]\times \mathbb{S}^{N-1}_\delta$ is flat, one can find an optimal transport $\mu_t$ such that $\mu_1=\delta_{x_1}$, $\mu_0=\mathcal H^N(A)\mathcal H^N|_A$, $d(x_1,A)\geq \frac{5}{4}\epsilon$ and $\mu_t=t^{N}\mathcal H^N(A)\mathcal H^N|_A$. 
If the $MCP(N,N+1)$ holds, then $\mu_t\geq \tau_{N, N+1}^{(t)}(\frac{5}{4}\epsilon)^{N+1}\mathcal H^N(A)\mathcal H^N|_{A}$.

Together with the previous inequality we see that the $MCP(N,N+1)$ cannot be satisfied.
\end{example}
\bibliography{new}
\bibliographystyle{amsalpha}}

\end{document}